\documentclass[a4paper,11pt,reqno]{amsart}
\usepackage[centering,includeheadfoot,margin=2.5 cm]{geometry}
\usepackage[percent]{overpic}
\usepackage{todonotes}
\usepackage{algpseudocode}
\usepackage{algorithm}
\usepackage{rotating}

\newcommand{\nodes}{{\mathcal A}}
\newcommand{\vphi}{\varphi}

\usepackage{graphics,enumitem,epsfig,textcomp}
\usepackage{amsfonts,amsmath,amssymb,euscript,color,mathrsfs}
\usepackage[utf8]{inputenc}	
\usepackage{float} 

\usepackage[caption=false,subrefformat=parens,labelformat=parens]{subfig}
\usepackage{url}
\usepackage[symbol]{footmisc}
\usepackage{hyperref}
\usepackage{cases}
\usepackage{bigints}
\hypersetup{
    colorlinks=true,
    linkcolor=blue,
    filecolor=blue,      
    urlcolor=blue,
    citecolor=blue,
    }

\usepackage[normalem]{ulem}
\usepackage[square,numbers]{natbib}
\usepackage{amsmath,amsthm,amssymb,amsfonts,lineno}
\newtheorem{pro}{Problem}

\numberwithin{equation}{section}
\newtheorem{theorem}{Theorem}

\newtheorem{remark}{Remark}

\def\diver{\mathrm{div}\,}

\def\d{\,\mathrm{d}}

\def\R{\mathbb{R}}
\def\C{\hbox{\rlap{\kern.24em\raise.1ex\hbox
      {\vrule height1.3ex width.9pt}}C}}
\def\P{\hbox{\rlap{I}\kern.16em P}}
\def\Q{\hbox{\rlap{\kern.24em\raise.1ex\hbox
      {\vrule height1.3ex width.9pt}}Q}}
\def\M{\hbox{\rlap{I}\kern.16em\rlap{I}M}}
\def\Z{\hbox{\rlap{Z}\kern.20em Z}}
\def\({\begin{eqnarray}}
\def\){\end{eqnarray}}
\def\[{\begin{eqnarray*}}
\def\]{\end{eqnarray*}}
\def\part#1#2{\frac{\partial #1}{\partial #2}}

\def\grad{\nabla}

\def\pmb#1{\setbox0=\hbox{$#1$}
  \kern-.025em\copy0\kern-\wd0
  \kern-.05em\copy0\kern-\wd0
  \kern-.025em\raise.0433em\box0 }

\def\d{\,\mathrm{d}}

\def\R{\mathbb{R}}

\def\P{\mathbb{P}}
\def\Q{\mathbb{Q}}
\def\C{\mathbb{C}}

\newcommand{\dx}{\mathrm{d}x}

\usepackage{cancel}
\title[Biological transportation structures with general entropy dissipation in 2D]{Self-regulated biological transportation structures with general entropy dissipation: \\ 2D case and leaf-shaped domain}

\begin{document}
\maketitle 

\centerline{
     {\large Clarissa Astuto}\footnote{Mathematical and Computer Sciences and Engineering Division,
         King Abdullah University of Science and Technology,
         Thuwal 23955-6900, Kingdom of Saudi Arabia;
         {\it clarissa.astuto@kaust.edu.sa}, and
         Department of Mathematics and Computer Sciences, University of Catania 95125, Italy;
         {\it clarissa.astuto@unict.dmi.it}} \qquad
     {\large Peter Markowich}\footnote{Mathematical and Computer Sciences and Engineering Division,
         King Abdullah University of Science and Technology,
         Thuwal 23955-6900, Kingdom of Saudi Arabia;
         {\it peter.markowich@kaust.edu.sa}, and
         Faculty of Mathematics, University of Vienna,
        Oskar-Morgenstern-Platz 1, 1090 Vienna;
         {\it peter.markowich@univie.ac.at}}\qquad
         {\large Simone Portaro}\footnote{Mathematical and Computer Sciences
            and Engineering Division,
         King Abdullah University of Science and Technology,
         Thuwal 23955-6900, Kingdom of Saudi Arabia;
         {\it simone.portaro@kaust.edu.sa}} \qquad {\large Giovanni Russo}\footnote{Department of Mathematics and Computer Sciences, University of Catania 95125, Italy;
         {\it russo@unict.dmi.it}}
     }

\section*{Abstract}
In recent years, the study of biological transportation networks has attracted significant interest, focusing on their self-regulating, demand-driven nature. This paper examines a mathematical model for these networks, featuring nonlinear elliptic equations for pressure and an auxiliary variable, and a reaction-diffusion parabolic equation for the conductivity tensor, introduced in \cite{portaro2022emergence}. The model, based on an energy functional with diffusive and metabolic terms, allows for various entropy generating functions, facilitating its application to different biological scenarios. We proved a local well-posedness result for the problem in Hölder spaces employing Schauder and semigroup theory. Then, after a suitable parameter reduction through scaling, we  computed the numerical solution for the proposed system using a recently developed ghost nodal finite element method \cite{astuto2024nodal}. An interesting aspect emerges when the solution is very articulated and the branches occupy a wide region of the domain.

\section{Introduction}
In recent years, there has been growing interest in understanding the principles and mechanisms that underpin the formation of biological transportation networks. Many biologists believe these networks are demand-driven by nature, developing without centralized control \cite{tero2010rules}, and finely tuned through cycles of evolutionary selection pressure. Thus, they can be seen as emergent structures arising from self-regulating processes. These structures, particularly the organization of leaf venation networks, vascular networks, and neural networks, have been a focal point for the biophysical and mathematical research communities, as evidenced by recent studies such as \cite{astuto2022comparison, astuto2023asymmetry, astuto2023finite, corson2010fluctuations, portaro2024measure, katifori2010damage}. Research in this area has significant implications, reaching beyond biology to fields like medicine, chemistry, and engineering.

In this work, we examine the elliptic-parabolic system introduced in \cite{portaro2022emergence}, which comprises two nonlinear elliptic equations for the pressure $p$ and the auxiliary variable $\sigma$, along with a reaction-diffusion parabolic equation for the conductivity tensor $\mathbb C$. This system seeks to describe the formation of biological network structures and is derived as the gradient flow of an energy functional (see \eqref{eq:energy}), incorporating a diffusive term, an activation term—interpreted as entropy dissipation—and a metabolic cost term.
The energy expression generalizes the one introduced in \cite{marko_perthame,marko_perthame_2}, where the convex entropy generating function $\Phi(p)$ was quadratic. The energy functional is constrained by a Poisson equation for pressure, which describes Darcy's law for slow flow in porous media, as discussed in \cite{darcy1856fontaines,neuman1977theoretical}. 
By specifying different forms of $\Phi(p)$, we can model various phenomena, such as Joule's heating and Fisher information \cite{portaro2022emergence}. This flexibility is crucial, as comparing known networks with newly discovered ones is a natural approach to studying biological networks. The introduction of $\Phi(p)$ is pivotal because it provides greater freedom in creating patterns that can be matched with empirical biological data. In future research, our model's outcomes can be rigorously compared with real biological data, enabling a more comprehensive analysis.

A one-dimensional analysis of the proposed model was conducted in \cite{astuto2023self}, which demonstrated local-in-time existence for smooth solutions and provided numerical evidence that $\C$ touches zero in finite time. In this paper, we focus on the multi-dimensional case, where the analysis is more challenging.

Furthermore, our work is dedicated to the mathematical modeling of biological transportation networks in porous media, specifically leaf venation in plants \cite{dengler2001vascular, malinowski2013understanding}. To this end, we define our equations in a leaf-shaped domain $\Omega$, contained within a rectangular region discretized by a regular square mesh. The mesh is intersected arbitrarily by $\Omega$. We propose a nodal ghost finite-element method, for a detailed explanation of the numerical scheme refer to \cite{astuto2024nodal}. The active mesh $\Omega_h$ is the subset of the tessellation that intersects $\Omega$. The finite element space on $\Omega_h$ is a restriction of the finite element space on the background mesh. The finite element method is based on a variational formulation over $\Omega$, with stabilization terms to handle instability from cut elements whose size is proportional to a power of the cell length. In \cite{astuto2024nodal,astuto2024comparison}, the authors introduce a "snapping back to grid" technique that aligns the snapping threshold with the penalization term, ensuring optimal convergence.

The paper is structured as follows: Section \ref{sec:model} introduces the mathematical model; Section \ref{sec:existence} establishes the existence of solutions to the system (\ref{eq:poisson}--\ref{eq:sigma_law}) using Hölder theory. Sections \ref{sec:variational_formulation} through \ref{sec:time_discretization} detail the numerical scheme, with special emphasis on space discretization in Section \ref{sec:space_discretization}. Section \ref{sec:numerical_results} presents various numerical tests in different geometries and with varying convex entropy generators. Finally, we conclude with some insights and future directions.

\section{Model}
\label{sec:model}
The recent study \cite{portaro2022emergence} introduced a novel category of self-regulatory processes. These processes are defined through the minimization of an entropy dissipation, expressed by the energy functional
\begin{align}
    \label{eq:energy}
    E[\mathbb{C}] = \int_{\Omega} \left(D^2 \frac{||\grad \mathbb{C}||^2}{2} + c^2 \Phi''(p) \nabla p \cdot ( \mathbb{C} + r \mathbb{I} ) \nabla p + \frac{\nu}{\gamma} ||\C||^{\gamma} \right) \mathrm{d}x.
\end{align}
This energy functional, denoted as $E=E[\mathbb{C}]$, undergoes minimization over all the $d$-dimensional symmetric and positive semidefinite conductivity tensor fields $\mathbb{C} = \mathbb{C}(x) \in \mathscr{S}^d(\R)$.
The diffusion coefficient $D^2$ controls the random effects in the transportation medium, while the activation parameter $c^2$ describes the tendency of the network to align with the pressure gradient.
The function $\Phi: \mathbb{R} \rightarrow \mathbb{R}$ represents a convex entropy generator, whereas $p = p(x)$ describes the scalar material pressure of the porous medium where the network arises. 
On the other hand, $r : \Omega \rightarrow \R^+$ describes the isotropic background permeability of the medium.
The metabolic cost to sustain the network is given by $\frac{\nu}{\gamma}|| \mathbb{C} ||^{\gamma}$ where $\nu>0$ is a metabolic coefficient, $\gamma > 0$ is the metabolic exponent and $|| \cdot ||$ denotes the Frobenius norm. For leaf venation in plants, $1/2\leq \gamma \leq 1$, as shown in  \cite{ hu2013optimization,hu2013adaptation}. {For this range of values, there is no proof of convergence to a unique steady state, while such a proof exists for $\gamma>1$ in \cite{marko_perthame_2}. Both cases have been numerically investigated in \cite{astuto2022comparison,astuto2023asymmetry}.} The domain $\Omega \subset \mathbb{R}^d$ (with $d \ge 1$) is assumed to be bounded, open and with $C^1$ boundary $\Gamma := \partial \Omega$.

The above energy functional is constrained by the mass conservation law
\begin{align}
    \label{eq:poisson}
    -\diver ( (\mathbb{C} + r \mathbb{I} ) \nabla p) = S \quad \text{in} \; \Omega,
\end{align}
where $S = S(x)$ denotes the given distribution of sources and sinks within the system, assumed to be constant over time. We impose that the material flux,
described by the expression ${(\mathbb{C}+r\mathbb{I} )\nabla p}$, is zero on $\Gamma$, i.e., 
\begin{equation}
(\mathbb{C} + r \mathbb{I} ) \nabla p \cdot \widehat n = 0 \quad \text{on} \; \Gamma. \notag
\end{equation}
Here, $\widehat n$ denotes the outer unit normal vector to $\Gamma$. Compatibility with such a choice requires
\begin{align*}
    \int_{\Omega} S(x) \mathrm{d}x = 0.
\end{align*}

The formal $L^2$-gradient flow for the energy functional \eqref{eq:energy}, constrained by the mass conservation law \eqref{eq:poisson}, reads
\begin{equation}
    \label{eq:gradient_flow}
    \frac{\partial \mathbb{C}}{\partial t} = D^2 \Delta \mathbb{C} + c^2 \left( \Phi''(p) \nabla p \otimes \nabla p + \frac{\nabla\sigma\otimes\nabla p + \nabla p\otimes\nabla\sigma}{2} \right) - \nu || \C ||^{\gamma-2} \C \quad \text{in} \; (0, \infty) \times \Omega,
\end{equation}
as detailed in Lemma 1 on \cite{portaro2022emergence}. The variable $t \geq 0$ serves as a temporal parameter for the gradient flow, and $\sigma = \sigma(t,x)$ solves the boundary value problem
\begin{equation}
    \label{eq:sigma_law}
    -\diver ( (\C + r \mathbb{I}) \nabla \sigma) = \Phi'''(p) \nabla p \cdot (\mathbb{\C} + r \mathbb{I}) \nabla p \quad \text{in} \, (0, \infty) \times \Omega,
\end{equation}
with the no-flux boundary condition $(\C + r \mathbb{I}) \nabla \sigma \cdot \widehat n = 0$ along $\Gamma$.
We choose homogeneous Neumann boundary conditions for all entries of $\mathbb{C}$ and the same type of conditions was selected for the numerical simulations carried out in \cite{hu2019optimization}. 
At the end, the system is closed by taking a positive definite initial condition $\C_0 > 0$ for the conductivity, i.e.,
\begin{align}
    \label{eq:ic}
    \C (t=0, x) = \C_{0} (x) \quad \mathrm{in} \; \Omega.
\end{align}

\begin{remark}
    We point out that in Lemma 1 of \cite{portaro2022emergence}, the boundary conditions for $p$ and $\sigma$ were specified as homogeneous Dirichlet, alongside a particular equilibrium condition for $\Phi(p)$ on the boundary, i.e., $\Phi'(p) = 0$ on $\Gamma$. However, transitioning to no-flux boundary conditions does not alter the gradient flow structure, thereby obviating the need for the equilibrium condition assumption for $\Phi(p)$.
\end{remark}

\section{Existence of solutions of the multi-dimensional problem in H\"older space}
\label{sec:existence}
The goal of this section is to show existence of a solution of the $d$-dimensional system (\ref{eq:poisson}--\ref{eq:sigma_law}) with initial condition \eqref{eq:ic}. For the sake of mathematical transparency, we restrict our analysis to the case of $r=0$, as regularization is not necessary for the subsequent theory. However, we mention that the following results can easily be carried over to the physical relevant case of positive background permeability $r$.
We shall work in the space of H\"older-continuous functions.

\begin{theorem} \label{thm:existence_uniqueness}
    Let $S \in C^{0,\alpha}(\Bar{\Omega})$ for some $\alpha \in (0, 1)$, $D=0$ and $\Phi \in C^4(\R)$.
    Then, the problem (\ref{eq:poisson}--\ref{eq:sigma_law}) subject to initial condition \eqref{eq:ic} - with $\C_0(x) \ge \beta$ - admits a local in time solution $\mathbb{C} \in X_T$ for some $T>0$, where the set $X_T$ is defined as
    \begin{align*}
         X_T := \left\{  \mathbb{C} \in C\left( [0, T]; C^{1, \alpha} \left(\Bar{\Omega}; \mathscr{S}^d(\R) \right)\right): \;  \mathbb{C} \ge \frac{\beta}{2} > 0, \| \mathbb{C} \|_{C( [0, T]; C^{1, \alpha} (\Bar{\Omega}; \mathscr{S}^d(\R) )} \le R \, \mathrm{ for \; some } \, R > 0  \right\}.
     \end{align*}
 \end{theorem}

\begin{proof}
     We shall employ the Banach fixed point theorem.
     For technical simplicity of the proof, we shall assume that $\mathbb{C}$ is a scalar quantity; however, the below calculations also hold in the tensor case with obvious modifications. Moreover, let $c^2 = \nu = 1$, as clearly, nothing changes in the more general case.
    
     For $T>0$, to be specified later, we construct the operator $\mathcal{S} : X_T \rightarrow X_T$ that maps $\Bar{\mathbb{C}} \in X_T$ to $\mathcal{S}(\Bar{\mathbb{C}}) =: \mathbb{C}$, whose fixed points are solutions of (\ref{eq:poisson}--\ref{eq:sigma_law}). Our goal is to demonstrate that $\mathcal{S}$ is a contraction on $X_T$ and that $\mathrm{Im}(\mathcal{S}) \subset X_T$. Fix $\Bar{\mathbb{C}} \in X_T$, then from \eqref{eq:poisson} we have 
     \begin{align*}
         - \Bar{\mathbb{C}} \Delta p - \nabla \Bar{\mathbb{C}} \cdot \nabla p = S,
     \end{align*}
     which by the Schauder theory \cite{gilbarg1977} admits a solution $p(t) \in C^{2, \alpha}(\Bar{\Omega})$ and there exists a constant $K(d, \alpha, \beta, R) =: K_{\beta, R} > 0$ such that
     \begin{align}
         \label{eq:holder_p}
         \| p(t) \|_{2, \alpha} \le K_{\beta, R} \| S \|_{0, \alpha},
     \end{align}
     where $\| \cdot \|_{C^{k,\alpha}(\Bar{\Omega})} = \| \cdot \|_{k,\alpha}$. For the sake of simplicity, we denote possibly different constants which depend on $\beta, R$ by the same symbol $K_{\beta, R}$. Observe that $ \Phi'''(p(t)) \Bar{\mathbb{C}}(t) | \nabla p(t) |^2 \in C^{0, \alpha}( \Bar{\Omega})$ and there exists $\sigma (t) \in C^{2, \alpha}(\Bar{\Omega})$ satisfying
     \begin{align*}
         \| \sigma(t) \|_{2, \alpha} &\le K_{\beta, R} \| \Phi''' (p) \|_{0, \alpha} \| \nabla p (t) \|_{0, \alpha}^2 \notag
     \end{align*}
     where we used the Banach algebra structure of Hölder spaces. It is straightforward to show that $\| \Phi'''(p) \|_{0, \alpha} \le \| \Phi^\textrm{\romannumeral 4} \|_{L^{\infty}(\mathrm{Im}(p))} \| p \|_{0, \alpha}$ where $\| \Phi^\textrm{\romannumeral 4} \|_{L^{\infty}(\mathrm{Im}(p))}$ just depends on $\beta, R$ so it will be successively incorporated in the constant $K_{\beta, R}$. Therefore, using \eqref{eq:holder_p} we rewrite the Hölder estimate for $\sigma$
     \begin{align}
         \label{eq:holder_sigma}
         \| \sigma (t) \|_{2, \alpha} \le K_{\beta, R} \| S \|_{0,\alpha}^3.
     \end{align}
     In the course of the computations, we shall omit the temporal dependence when there is no risk of confusion. Next, let $\Bar{\mathbb{C}_1}, \Bar{\mathbb{C}_2} \in X_T$, $p_i := p[\Bar{\mathbb{C}_i}]$ and $\sigma_i := \sigma[\Bar{\mathbb{C}_i}], \, i = 1,2$ be the solutions of \eqref{eq:poisson}, \eqref{eq:sigma_law}, respectively. Then, integrating \eqref{eq:gradient_flow} on $t \in [0, T]$ and subtracting we get
    \begin{align}
        \label{eq:contraction_computation_1}
        \| \C_1 - \C_2 \|_{1, \alpha} &\le \bigg{\|} \int_0^t \bigg[ \left( \Phi''(p_1) | \nabla p_1 |^2 - \Phi''(p_2) | \nabla p_2 |^2  \right) + \left( \nabla p_1 \cdot \nabla \sigma_1 - \nabla p_2 \cdot \nabla \sigma_2 \right) \notag \\
        &\quad \quad \qquad + \left( |\C_1|^{\gamma-1} - |\C_2|^{\gamma-1} \right) \bigg] \, \d s \bigg{\|}_{1, \alpha} \notag \\
        &\le \int_0^t \left( J_1 (s) + J_2(s) + J_3(s) \right) \, \d s,
    \end{align}
     where
     \begin{align}
         J_1 (s) &:= \| \Phi''(p_1) | \nabla p_1 |^2 - \Phi''(p_2) | \nabla p_2 |^2 \|_{1, \alpha} \notag \\
         J_2(s) &:= \| \nabla p_1 \cdot \nabla \sigma_1 - \nabla p_2 \cdot \nabla \sigma_2 \|_{1, \alpha} \notag \\
         J_3(s) &:= \| |\C_1|^{\gamma-1} - |\C_2|^{\gamma-1} \|_{1, \alpha}. \notag
     \end{align}
     Firstly, adding and subtracting $\Phi''(p_2) | \nabla p_1 |^2$, using \eqref{eq:holder_p}, the local Lipschitz continuity of $\Phi''$ and the continuous embedding of $C^{2, \alpha}(\Bar{\Omega})$ into $C^{1, \alpha}(\Bar{\Omega})$, we estimate $J_1(s)$
    \begin{align}
     \label{eq:J1_estimate1}
         J_1 (s) &\le \big{\|} \left( \Phi''(p_1) - \Phi''(p_2) \right) | \nabla p_1 |^2 \big{\|}_{1, \alpha} + \big{\|} \Phi''(p_2) \left( | \nabla p_1 |^2 - | \nabla p_2 |^2 \right) \big{\|}_{1, \alpha} \notag \\
         &\le K_{\beta, R} \| S \|_{0, \alpha}^2 \| p_1 - p_2 \|_{1, \alpha} + \| \Phi''(p_2) \|_{1, \alpha} \| \nabla p_1 + \nabla p_2 \|_{1, \alpha} \| \nabla p_1 - \nabla p_2 \|_{1, \alpha} \notag \\
         &\le K_{\beta, R} \| S \|_{0, \alpha}^2 \| p_1 - p_2 \|_{2, \alpha}.
     \end{align}
     To estimate $p_1 - p_2$, we consider the elliptic equation \eqref{eq:poisson} for both $p_1, p_2$ and subtract the two, obtaining
     \begin{align*}
         - \diver \left( \Bar{\C_1} (\nabla p_1 - \nabla p_2) \right) = \diver \left( \left((\Bar{\C_1} - \Bar{\C_2}\right) \nabla p_2 \right) \in C\left([0,T]; C^{0, \alpha}(\Bar{\Omega})\right).
     \end{align*}
     Employing Schauder theory, estimate \eqref{eq:holder_p} and the continuous embedding of $C^{1, \alpha}(\Bar{\Omega})$ into $C^{0, \alpha}(\Bar{\Omega})$, we can write
     \begin{align}
         \label{eq:holder_p1-p2}
         \| p_1 - p_2 \|_{2, \alpha} \le K_{\beta, R} \| S \|_{0, \alpha} \| \Bar{\C_1} - \Bar{\C_2} \|_{1, \alpha},
     \end{align}
     which substituted into \eqref{eq:J1_estimate1} yields
     \begin{align}
         \label{eq:J1_estimate2}
         J_1(s) \le K_{\beta, R} \| S \|_{0, \alpha}^3 \| \Bar{\C_1} - \Bar{\C_2} \|_{1, \alpha} \quad \forall \; \Bar{\C_1}, \Bar{\C_2} \in X_T.
     \end{align}
     Adding and subtracting $\nabla p_1 \cdot \nabla \sigma_2$ and using \eqref{eq:holder_p}, \eqref{eq:holder_sigma}, \eqref{eq:holder_p1-p2} we obtain the following estimate for $J_2(s)$
     \begin{align}
         \label{eq:J2_estimate}
         J_2 (s) &\le \big{\|} \nabla p_1 \cdot \left( \nabla \sigma_1 - \nabla \sigma_2 \right) \big{\|}_{1, \alpha} + \big{\|} \nabla \sigma_2 \cdot \left( \nabla p_1 - \nabla p_2 \right) \big{\|}_{1, \alpha} \notag \\
         &\le K_{\beta, R} \| S \|_{0, \alpha} \left( \| \sigma_1 - \sigma_2 \|_{2, \alpha} + \| S \|_{0, \alpha}^3 \| \Bar{\C_1} - \Bar{\C_2} \|_{1, \alpha} \right).
     \end{align}
     An estimate of $\sigma_1 - \sigma_2$ is obtained by considering the elliptic equation \eqref{eq:sigma_law} for $\sigma_1, \sigma_2$ and subtracting the two equations to get
     \begin{align*}
         - \diver \left( \Bar{\C_1} \nabla \sigma_1 - \Bar{\C_2} \nabla \sigma_2 \right) = \Phi'''(p_1) \Bar{\C_1} | \nabla p_1 |^2 - \Phi'''(p_2) \Bar{\C_2} | \nabla p_2 |^2,
     \end{align*}
     which is recast, after some algebraic manipulation, into
     \begin{align*}
         - \diver \left( \Bar{\C_1} \left( \nabla \sigma_1 - \nabla \sigma_2 \right) \right) = &\diver \left( \left( \Bar{\C_2} - \Bar{\C_1} \right) \grad \sigma_2 \right) + \left( \Phi'''(p_1) - \Phi'''(p_2) \right) \Bar{\C_1} | \grad p_1 |^2 \\
         &+ \Phi'''(p_2) \big[ \left( \Bar{\C_1} - \Bar{\C_2} \right) | \grad p_1 |^2 + \Bar{\C_2} \left( | \grad p_1 | + | \grad p_2 | \right) \left( | \grad p_1 | - | \grad p_2 | \right)\big].
     \end{align*}
     The right-hand side of the latter belongs to $C\left([0,T]; C^{0, \alpha}(\Bar{\Omega})\right)$. Therefore, the Schauder theory gives the desired estimate
     \begin{align}
         \label{eq:holder_sigma1-2}
         \| \sigma_1 - \sigma_2 \|_{2, \alpha} \le K_{\beta, R} \| S \|_{0, \alpha}^3 \| \Bar{\C_1} - \Bar{\C_2} \|_{1, \alpha}.
     \end{align}
     Applying \eqref{eq:holder_sigma1-2} to \eqref{eq:J2_estimate} we get
     \begin{align}
        \label{eq:J2_estimate2}
        J_2 (s) \le K_{\beta, R} \| S \|_{0, \alpha}^4 \| \Bar{\C_1} - \Bar{\C_2} \|_{1, \alpha} \quad \forall \; \Bar{\C_1}, \Bar{\C_2} \in X_T.
     \end{align}
     Regarding $J_3(s)$, we observe that the derivative of the function $x \rightarrow |\Bar{\C}(x)|^{\gamma-1}$ is locally bounded away from the zero. Given that $\Bar{\C} \in X_T$ implies $\Bar{\C} \ge \frac{\beta}{2} > 0$, it is not feasible to consider a neighborhood around zero. This condition implies that the function is locally Lipschitz continuous, that is, there exists a constant $K_\beta > 0$ such that
     \begin{align}
         \label{eq:J3_estimate}
         J_3(s) \le K_\beta \| \Bar{\C_1} - \Bar{\C_2} \|_{1, \alpha} \quad \forall \; \Bar{\C_1}, \Bar{\C_2} \in X_T. 
     \end{align}
     Substitution of \eqref{eq:J1_estimate2}, \eqref{eq:J2_estimate2} and \eqref{eq:J3_estimate} into \eqref{eq:contraction_computation_1} yields
     \begin{align*}
         \| \C_1 - \C_2 \|_{1, \alpha} \le K_{\beta, R} \left( \| S \|_{0, \alpha}^4 + \| S \|_{0, \alpha}^3 + 1 \right) \int_0^t \| \Bar{\C_1} (s) - \Bar{\C_2} (s) \|_{1, \alpha} \d s.
     \end{align*}
     Subsequently, taking the maximum on $[0, T]$ of the latter expression, we get
     \begin{align*}
         \| \C_1 - \C_2 \|_{X_T} \le K_{\beta, R} \left( \| S \|_{0, \alpha}^4 + \| S \|_{0, \alpha}^3 + 1 \right) T \| \Bar{\C_1} - \Bar{\C_2} \|_{X_T},
     \end{align*}
     and choosing $T$ such that
     \begin{align*}
         T < \frac{1}{K_{\beta, R} \left( \| S \|_{0, \alpha}^4 + \| S \|_{0, \alpha}^3 + 1 \right)}
     \end{align*}
     we prove the contraction property for $\mathcal{S}$.
    
     To complete the proof, it remains to show that $\mathcal{S}(\Bar{\C}) =: \C \in X_T$ for every $\Bar{\C} \in X_T$.
     Integrating \eqref{eq:gradient_flow} on $[0, t]$ for some $t \in (0, T)$ and using \eqref{eq:holder_p}, \eqref{eq:holder_sigma} we obtain the following bound from below for $\C$
     \begin{align*}
         \C &\ge \beta - \int_0^t \big{|} \Phi''(p) | \grad p |^2 + \grad p \cdot \grad \sigma - |\C|^{\gamma-1} \big{|} \d s \\
         &\ge \beta - \int_0^t \| \Phi''(p) \|_{0, \alpha} \| \grad p \|^2_{0, \alpha} + \| \grad p \|_{0, \alpha} \| \grad \sigma \|_{0, \alpha} + \| |\C|^{\gamma-1} \|_{0, \alpha} \d s \\
         &\ge \beta - K_{\beta, R} \left( \| S \|_{0, \alpha}^4 + \| S \|_{0, \alpha}^3 + 1 \right) T,
     \end{align*}
     where for the last term we used the fact that $\| \C \|_{0,\alpha} \le R$. 
     Designing $T$ such that
     \begin{align*}
         T \le \frac{\beta}{2 K_{\beta, R} \left( \| S \|_{0, \alpha}^4 + \| S \|_{0, \alpha}^3 + 1 \right) },
     \end{align*}
     we ensure $\C \ge \frac{\beta}{2}$, so that $\textrm{Im}(\mathcal{S}) \subset X_T$.
     Finally, existence and uniqueness follow from Banach's fixed point theorem.
\end{proof}

\begin{remark}
    In Theorem \ref{thm:existence_uniqueness}, we initially assumed $D=0$. However, this restriction can be relaxed to $D>0$ by applying semigroup theory. Indeed, we notice from the above proof that the function
    \begin{align*}
        f(\C(t,x)) := c^2 \left(  \Phi''(p) \nabla p \otimes \nabla p + \frac{\nabla\sigma\otimes\nabla p + \nabla p\otimes\nabla\sigma}{2} \right) - \nu || \C ||^{\gamma-2} \C
    \end{align*}
    is locally Lipschitz with respect to $\C$. Moreover, the operator $\Delta$ is the infinitesimal generator of a $C_0$ semigroup $\mathscr T(t)$ on $C^{1,\alpha}(\overline{\Omega};\mathscr{S}^d(\R))$. Then, by semigroup theory \cite[Section 6.1, Theorem 1.4]{pazy2012semigroups}, for every $\C_0(x) \in C^{1,\alpha}(\overline{\Omega};\mathscr{S}^d(\R))$ there exists a maximal time $t_{\mathrm{max}}<\infty$ for which the initial value problem (\ref{eq:gradient_flow}-\ref{eq:ic}) admits a unique mild solution $\C$ on the time interval $[0, t_{\mathrm{max}}[$, i.e.,
    \begin{align*}
        \C(t,x) = \mathscr T(t) \C_0(x) + \int_0^t \mathscr T(t-s) f(\C(s,x)) ds.
    \end{align*}
    Furthermore, it is shown that $\lim_{t \rightarrow t_{\mathrm{max}}} \| \C(t) \|_{1, \alpha} = \infty$.
\end{remark}

Notice that the energy expression defined in \eqref{eq:energy} is dependent on four parameters, namely, $E[\C] \equiv E_{D, c, \nu, \gamma}[\C]$. As the objective of this paper is to conduct several numerical experiments, we propose a parameter reduction through scaling. It is observed that $1/D$ and $c^2$ exert similar influences on the system dynamics. Specifically, as these parameters increase, the network number of branches increases \cite{marko_albi}. Consequently, we introduce the following scaling aimed at simplifying the numerical experiments
\begin{align}
\label{eq:energy_scaled}
    E_{D, c, \nu, \gamma}[\C] &= c^2 \int_{\Omega} \left( \frac{D^2}{c^2} \frac{||\nabla \C ||^2}{2} + \Phi''(p) \nabla p \cdot \left( \C + r \mathbb{I} \right) \nabla p + \frac{\nu}{c^2 \gamma} || \C ||^{\gamma} \right) \dx \notag \\
    &=: c^2 E_{\widetilde{D}, \widetilde{\nu}, \gamma}[\C]
\end{align}
where $\widetilde{D} := \frac{D}{c}$ and $\widetilde{\nu} = \frac{\nu}{c^2}$. The rescaled energy $E_{\widetilde{D}, \widetilde{\nu}, \gamma}[\C]$ modifies the original formulation such that the kinetic term is normalized. In contrast, the diffusion term becomes comparatively small (considering $D$ as a small parameter and $c$ as a large one), the metabolic term can be adjusted according to the specific experimental setup of interest, and the rescaled time variable is $\widetilde t = c^2 t$, that we rename as $t$ for simplicity.

\section{Variational formulation}
\label{sec:variational_formulation}
Let $\Omega\subset R$ be an open, bounded and sufficiently smooth domain in $\mathbb R^2$, 
where $R$ is a rectangular region. We now rewrite  (\ref{eq:poisson}-\ref{eq:ic}) according to the scaling defined in \eqref{eq:energy_scaled}
\begin{subequations}
\begin{align}
\label{eq_darcy_p_tens}
    &-{\rm div} \left(\left(\mathbb{C} + r \mathbb{I} \right) \nabla p \right) = S \qquad && {\rm in }\,(0, \infty) \times \Omega\\
    \label{eq_darcy_sigma_tens}
    & -{\rm div} \left( \left(\mathbb{C} + r \mathbb{I} \right) \nabla \sigma \right) = \Phi'''(p)\nabla p\cdot\left(\mathbb{C} + r \mathbb{I} \right)\nabla p \qquad && {\rm in }\, (0, \infty) \times \Omega \\
    \label{eq_reaction_diff_tens}
  &\frac{\partial \mathbb{C}}{\partial t} = \widetilde D^2\Delta \mathbb{C} + \Phi''(p)\nabla p\otimes \nabla p 
 + \frac{\nabla\sigma\otimes\nabla p + \nabla p\otimes\nabla\sigma}{2} -  \widetilde \nu||\mathbb{C}||^{\gamma - 2}\mathbb{C}, \quad && {\rm in }\,(0, \infty) \times \Omega,
\end{align}
\end{subequations}
together with the boundary conditions 
\begin{equation}
\label{eq_bc}
     \left(\mathbb{C} + r \mathbb{I} \right) \nabla p \cdot \widehat n = 0, \quad  \left(\mathbb{C} + r \mathbb{I} \right)\nabla \sigma \cdot \widehat n = 0, \quad \nabla \mathbb{C}\cdot \widehat n = 0 , \quad \text{for } x\in \Gamma, \, t\geq 0.
\end{equation}
For simplicity, in the equation we omit the dependence on the space and time variables of the functions.

The variables $p$ and $\sigma$ are  solutions of a diffusion equation with zero flux boundary conditions, consequently we impose a null-mean condition that guarantees uniqueness of the solutions. Therefore, we introduce the space $Q$ defined as
\begin{equation}
    Q = \biggl\{ q\in H^1(\Omega) : \int_\Omega q\,{ \rm d}x = 0\biggr\}.
\end{equation}
We also introduce the following tensor-valued space \textbf{V}, which takes into account the symmetry and positive semidefiniteness of the tensor $\mathbb C$, and is defined as follows
\begin{equation}
    \mathbf{V} = \biggl\{ \mathbb{B}\in[H^1(\Omega)]^{2\times2}:\mathbb{B}=\mathbb{B}^\top, \, \mathbb B\geq 0 \biggr\}.
\end{equation}
Multiplying (\ref{eq_darcy_p_tens}--\ref{eq_darcy_sigma_tens}) by test functions $q, v \in Q$, taking the contraction of \eqref{eq_reaction_diff_tens} against a test tensor $\mathbb{B}\in\mathbf{V}$, and integrating over $\Omega$, we obtain

\begin{subequations}
\begin{align}
&-\int_\Omega {\rm div} \left( (\mathbb{C}+r\mathbb{I}) \nabla p \right)\, q\,{ \rm d}x  = \int_\Omega S\, q\,{ \rm d}x && \forall q\in Q \\
&-\int_\Omega {\rm div} \left( (\mathbb{C}+r\mathbb{I}) \nabla \sigma \right)\, v \,{ \rm d}x  = \int_\Omega \Phi'''(p)\nabla p\cdot\left(\mathbb{C}+r\mathbb{I}\right)\nabla p\, v\,{ \rm d}x && \forall v\in Q \\
&\int_\Omega \frac{\partial \mathbb{C}}{\partial t} : \mathbb{B}\,{ \rm d}x - \widetilde D^2\int_\Omega (\Delta \mathbb{C}) : \mathbb{B}\,{ \rm d}x + \widetilde \nu \int_\Omega ||\mathbb{C}||^{\gamma - 2} \mathbb{C} : \mathbb{B}\,{ \rm d}x \\
& \notag \hspace{1.5cm} -\int_\Omega \left( \Phi''(p) \nabla p \otimes \nabla p + \frac{\nabla\sigma\otimes\nabla p + \nabla p\otimes\nabla\sigma}{2}\right) : \mathbb{B}\,{ \rm d}x  = 0 && \forall \mathbb{B}\in \mathbf{V}.
\end{align}
\end{subequations}
Integrating by parts and taking into account the boundary conditions, we have
\begin{subequations}
\begin{align}
&\int_\Omega \nabla q \cdot (\mathbb{C}+r\mathbb{I}) \nabla p \,{ \rm d}x  = \int_\Omega S\, q\,{ \rm d}x && \forall q\in Q \\
&\int_\Omega \nabla v \cdot (\mathbb{C}+r\mathbb{I}) \nabla \sigma \,{ \rm d}x  = \int_\Omega \Phi'''(p)\nabla p\cdot\left(\mathbb{C}+r\mathbb{I}\right)\nabla p\, v\,{ \rm d}x && \forall v\in Q \\
&\int_\Omega \frac{\partial \mathbb{C}}{\partial t} : \mathbb{B}\,{ \rm d}x + \widetilde D^2\int_\Omega \nabla \mathbb{C} : \nabla\mathbb{B}\,{ \rm d}x +  \widetilde \nu \int_\Omega ||\mathbb{C}||^{\gamma - 2} \mathbb{C} : \mathbb{B}\,{ \rm d}x \\
&\notag -  \int_\Omega \left( \Phi''(p)\nabla p \otimes \nabla p + \frac{\nabla\sigma\otimes\nabla p + \nabla p\otimes\nabla\sigma}{2} \right) : \mathbb{B}\,{ \rm d}x    = 0 && \forall \mathbb{B}\in \mathbf{V}.
\end{align}
\end{subequations}
Finally, we denote the scalar product in $L^2(\Omega)$ and its natural extension to $2$-dimensional tensors in the space of symmetric positive semidefinite tensors $\mathscr{S}^2(\R)$, i.e., $L^2(\Omega; \mathscr{S}^2(\R))$, using the same notation $(\cdot,\cdot)$ to simplify and reduce the notational complexity. Consequently, our problem, when expressed in its variational formulation, is as follows.

\begin{pro}\label{pro:variational}
Given $S \in L^2(\Omega)$ and $\mathbb{C}_0\in\mathbf{V}$, find $p,\sigma\in Q$ and $\mathbb{C} \in \mathbf{V}$ such that, for almost every $t\in(0,T)$, it holds
\begin{subequations}
\begin{align}
\left( \nabla q, (\mathbb{C}+r\mathbb{I}) \nabla p  \right) & = \left( S,q \right) && \forall q\in Q \\
\left( \nabla v, (\mathbb{C}+r\mathbb{I}) \nabla \sigma \right)  & = \left( \Phi'''(p)\nabla p\cdot\left(\mathbb{C} + r \mathbb{I} \right)\nabla p,v \right) && \forall v\in Q \\
\label{eq:variational_ev}\left(\frac{\partial \mathbb{C}}{\partial t},\mathbb{B}\right)  - &  \left(\Phi''(p) \nabla p \otimes \nabla p+ \frac{\nabla\sigma\otimes\nabla p + \nabla p\otimes\nabla\sigma}{2},\mathbb{B}\right)\\
\notag & = - \widetilde D^2\left(\nabla \mathbb{C},\nabla\mathbb{B}\right) -  \widetilde \nu \left( || \mathbb{C} ||^{\gamma - 2} \mathbb{C}(t) , \mathbb{B}\right) && \forall \mathbb{B}\in \mathbf{V} \label{eq:CB_cont}\\
 (\mathbb{C}+r\mathbb{I})\nabla p \cdot \widehat n &= 0 && \text{on } \Gamma \notag\\
 (\mathbb{C}+r\mathbb{I})\nabla \sigma \cdot \widehat n &= 0 && \text{on } \Gamma \notag\\
 \nabla \mathbb{C}\cdot \widehat n & = 0 && \text{on } \Gamma\\
\mathbb{C}(0,x) & = \mathbb{C}_0(x) && \text{in } \Omega.
\end{align}
\end{subequations}
\end{pro}
We remark that by the scalar product of two tensors in \eqref{eq:variational_ev}, we mean 
\begin{eqnarray*} \left(\frac{\partial {C}_{ij}}{\partial t},{B}_{ij}\right)  -  \left(\Phi''(p) \partial_{x_i} p \, \partial_{x_j} p + \frac{\partial_{x_i} \sigma \, \partial_{x_j} p +\partial_{x_i} p \, \partial_{x_j} \sigma}{2} ,{B}_{ij}\right)
\notag \\ = - \widetilde D^2\left(\nabla {C}_{ij}(t),\nabla {B}_{ij}\right) -  \widetilde \nu \left( || \mathbb{C}(t) ||^{\gamma - 2} {C}_{ij}(t) , {B}_{ij}\right), \quad i,j = 1,2,  \end{eqnarray*}
where $C_{ij}$ and $B_{ij}$ are the components of the tensors $\mathbb C$ and $\mathbb B$, respectively.

\section{Space discretization}
\label{sec:space_discretization}
{We assume the region $R$ is the unit square $R = [0,1]^2$. We discretize it}  by a uniform mesh of squares of size $h=1/N$, where $N \in \mathbb N$ is the number of cells in each space direction (see Fig.~\ref{fig:Domain2D} (left panel)). Following the approach showed in  \cite{Osher2002,Russo2000,book:72748, Sussman1994}, the domain $\Omega$ is implicitly defined by a level set function $\phi(x,y)$ that is negative inside $\Omega$, positive in $R\setminus \Omega$ and zero on the boundary $\Gamma$:
\begin{eqnarray}
	\Omega = \{(x,y): \phi(x,y) < 0\}, \qquad
	\Gamma = \{(x,y): \phi(x,y) = 0\}.
\end{eqnarray}
The outer unit normal vector $\widehat n$ in \eqref{eq_bc} can be computed as $\widehat n = \frac{\nabla \phi }{|\nabla \phi|}$.
\begin{figure}[H]
    \centering
    \begin{minipage}{.49\textwidth}
\centering\begin{overpic}[abs,width=0.85\textwidth,unit=1mm,scale=.25]{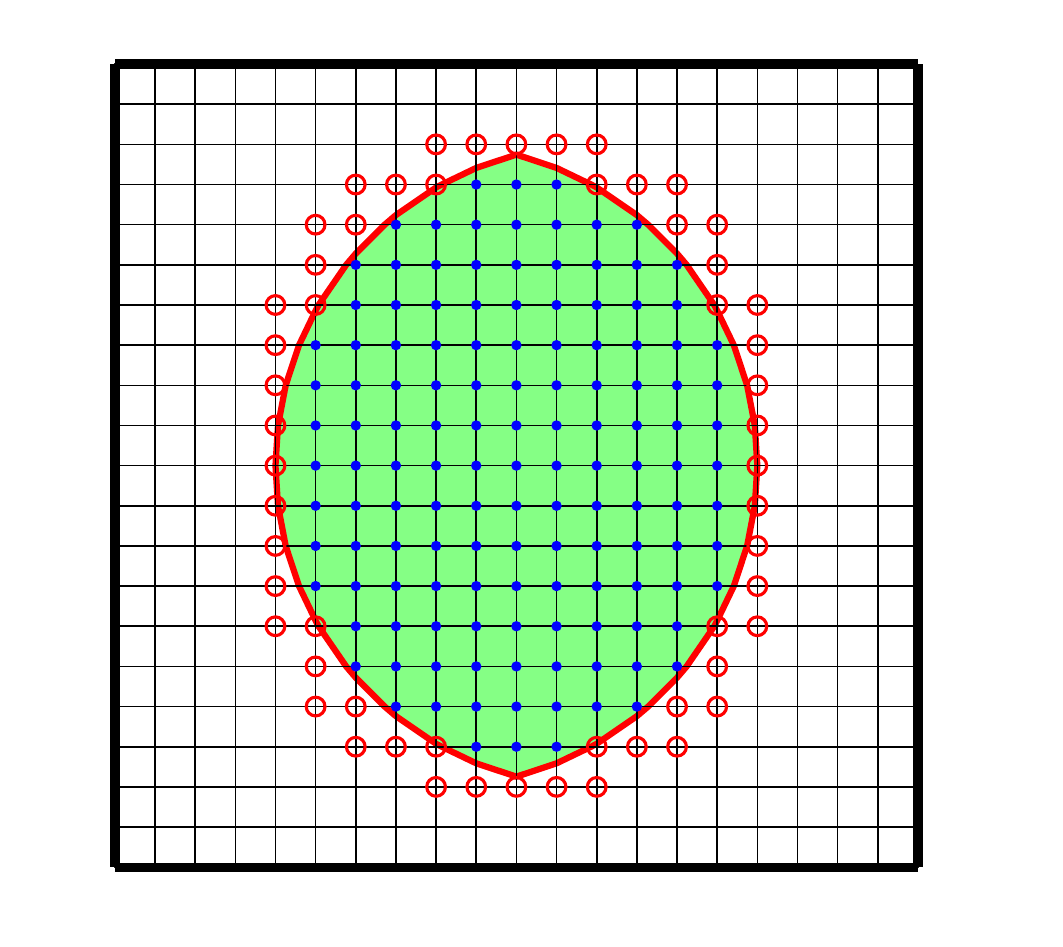}        
\put(-2,50){(a)}
\end{overpic}
    \end{minipage}
    \begin{minipage}{.49\textwidth}
\centering\begin{overpic}[abs,width=0.85\textwidth,unit=1mm,scale=.25]{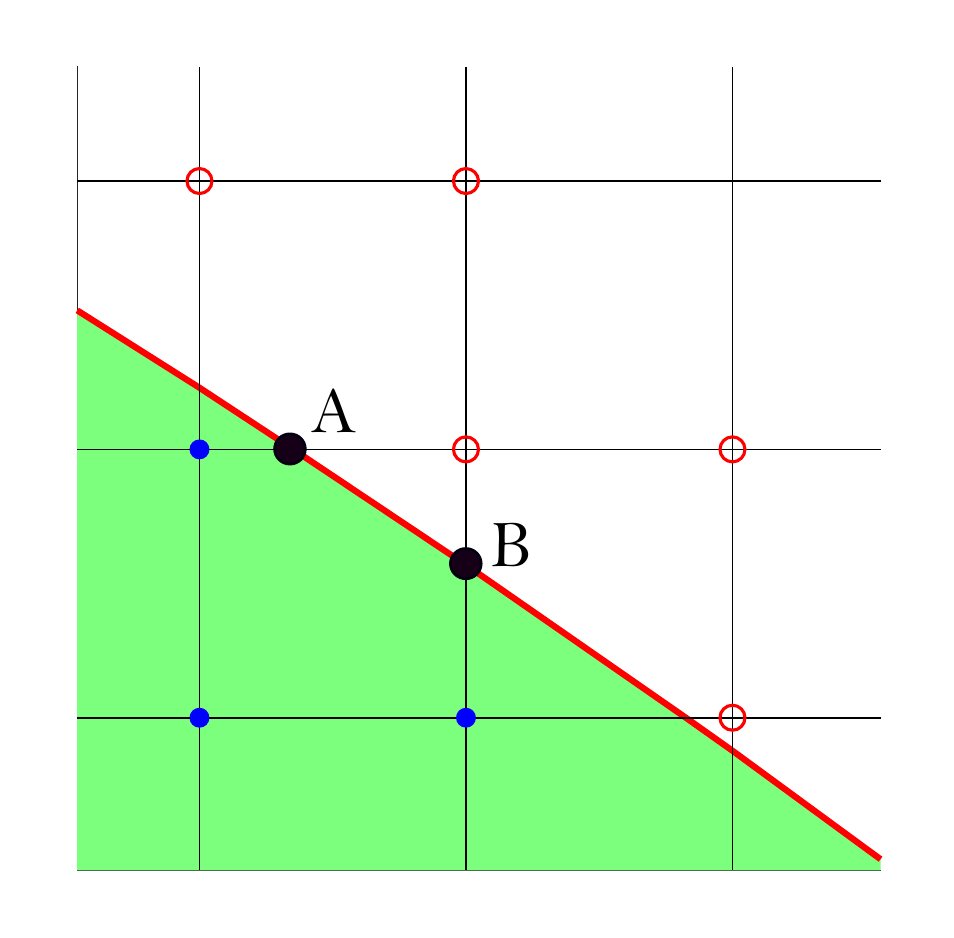}        
\put(-2,54){(b)}
\end{overpic}
    \end{minipage}
\caption{\textit{Discretization of the computational domain. $\Omega$ is the green region inside the unit square $R$. (a): classification of the grid points: the blue points are the internal ones while the red circles denote the ghost points. (b): points of intersection between the grid and the boundary $\Gamma$ (see the definition of $A$ and $B$ in Algorithm \ref{alg_ab}).}}  
\label{fig:Domain2D}
\end{figure}


The set of grid points will be denoted by $\mathcal N$, with $\# \mathcal N = (1+N)^2$, the active nodes (i.e.,\ internal $\mathcal{I}$ or ghost $\mathcal{G}$) by $\mathcal A = \mathcal{I}\cup\mathcal{G} \subset \mathcal N$, the set of inactive points by $\mathcal O \subset \mathcal N$, with $\mathcal O\cup\mathcal A = \mathcal N$ and $\mathcal O \cap \mathcal A = \emptyset$ and the set of cells by $\mathcal C$, with $\# \mathcal C = N^2$. Finally, we denote by $\Omega_c = R\setminus \Omega$ the outer region in $R$.

Here, we define the set of ghost points $\mathcal{G}$, which are grid points that belong to $\Omega_c$, with at least an internal point as neighbor, formally defined as
\begin{equation}
\notag
	(x,y) \in \mathcal{G} \iff (x,y) \in {\mathcal N}\cap \Omega_c  \text{ and } \{(x \pm h,y),(x,y\pm h), (x \pm h,y\pm h) \} \cap \mathcal I \neq \emptyset.
\end{equation}


The discrete spaces $Q_h$ and $\textbf{V}_h$ are given by the piecewise bilinear functions which are continuous in $R$.
As a basis of $Q_h$, we choose the following functions:
\begin{equation}
    \varphi_{i}(x,y) = \max\left\{
        \left(1-\frac{|x-x_i|}{h}\right) 
        \left(1-\frac{|y-y_i|}{h}\right),0
    \right\},
    \label{eq:V_h2}
\end{equation}
with $i = (i_1,i_2)$ an index that identifies a node on the grid. In the case of the symmetric positive semidefinite tensor $B_h \in \textbf{V}_h$, the basis for the four components are also defined in \eqref{eq:V_h2}. The generic element $u_h\in Q_h$ will have the following representation
\begin{equation}
    u_h(x,y) = \sum_{i\in\nodes}u_i\vphi_i(x,y).
    \label{eq:u_h2}
\end{equation}
To solve the variational problem \eqref{pro:variational}, we employ a finite-dimensional discetization. Specifically, the functions $q$ and $v$, which are originally defined in the function space $Q$, are approximated by functions $u_h$ and $v_h$ respectively, both of which belong to a finite-dimensional subspace $Q_h$. Additionally, the operator $\mathbb B$ is replaced by its discrete counterpart $\mathbb B_h$, acting within the space ${\bf V}_h$. 
To perform computations, the  domain $\Omega$ is approximated by a polygonal domain $\Omega_h$. This approximation also extends to the boundary $\Gamma$, which is represented by $\Gamma_h$. Consequently, the original integrals defined over $\Omega$ and its boundary $\Gamma$ are now evaluated over $\Omega_h$ and $\Gamma_h$, respectively.

\begin{pro}\label{pro:variational_2}
Given $S_h \in L^2(\Omega)$ and $\mathbb{C}_{0,h}\in\mathbf{V}_h$, find $p_h$, $\sigma_h\in Q_h$ and $\mathbb{C}_h \in \mathbf{V}_h$ such that, for almost every $t\in(0,T)$, it holds
\begin{subequations}
\begin{align}
\label{eq:p_disc}
\left(\nabla q_h, (\mathbb{C}_h+r\mathbb{I}_h) \nabla p_h\right)  &= \left( S_h, q_h\right) & \forall q_h\in Q_h \\ \label{eq:sigma_disc}
\left( \nabla v_h,(\mathbb{C}_h+r\mathbb{I}_h) \nabla \sigma_h\right) & = \left( \Phi'''(p_h)\nabla p_h\cdot\left(\mathbb{C}_h + r\mathbb I_h\right)\nabla p_h
, v_h\right) &\forall v_h\in Q_h 
 \\ \nonumber  
\left( \frac{\partial \mathbb{C}_h}{\partial t} ,\mathbb{B}_h\right)  - \left(\Phi''(p_h) \right. & \left. \nabla p_h \otimes \nabla p_h + \frac{\nabla \sigma_h \otimes \nabla p_h + \nabla p_h \otimes \nabla \sigma_h}{2},\mathbb{B}_h\right) &\\ \label{eq:C_disc} & = - \widetilde D^2\left(\nabla \mathbb{C}_h, \nabla\mathbb{B}_h\right) -  \widetilde \nu \left( || \mathbb{C}_h ||^{\gamma - 2} \mathbb{C}_h , \mathbb{B}_h\right) & \forall \mathbb{B}_h\in \mathbf{V}_h
\end{align}
\end{subequations}
under the following boundary conditions
\begin{subequations}
\begin{align} 
\left(\mathbb C_h + r\mathbb I_h \right)\nabla p_h \cdot \widehat n &= 0 & \text{on } \Gamma_h \\ \nonumber
 \left(\mathbb C_h + r\mathbb I_h \right)\nabla \sigma_h \cdot \widehat n &= 0 & \text{on } \Gamma_h
 \\ \nonumber
 \nabla \mathbb{C}_h\cdot \widehat n & = 0 & \text{on } \Gamma_h
\\  \mathbb{C}_h(t 
 = 0) & = \mathbb{C}_{0,h} & \text{in }  \Omega_h.
\end{align}
\end{subequations}
\end{pro}


\begin{figure}[H]
    \centering
    \begin{minipage}{.4\textwidth}
\begin{overpic}[abs,width=\textwidth,unit=1mm,scale=.25]{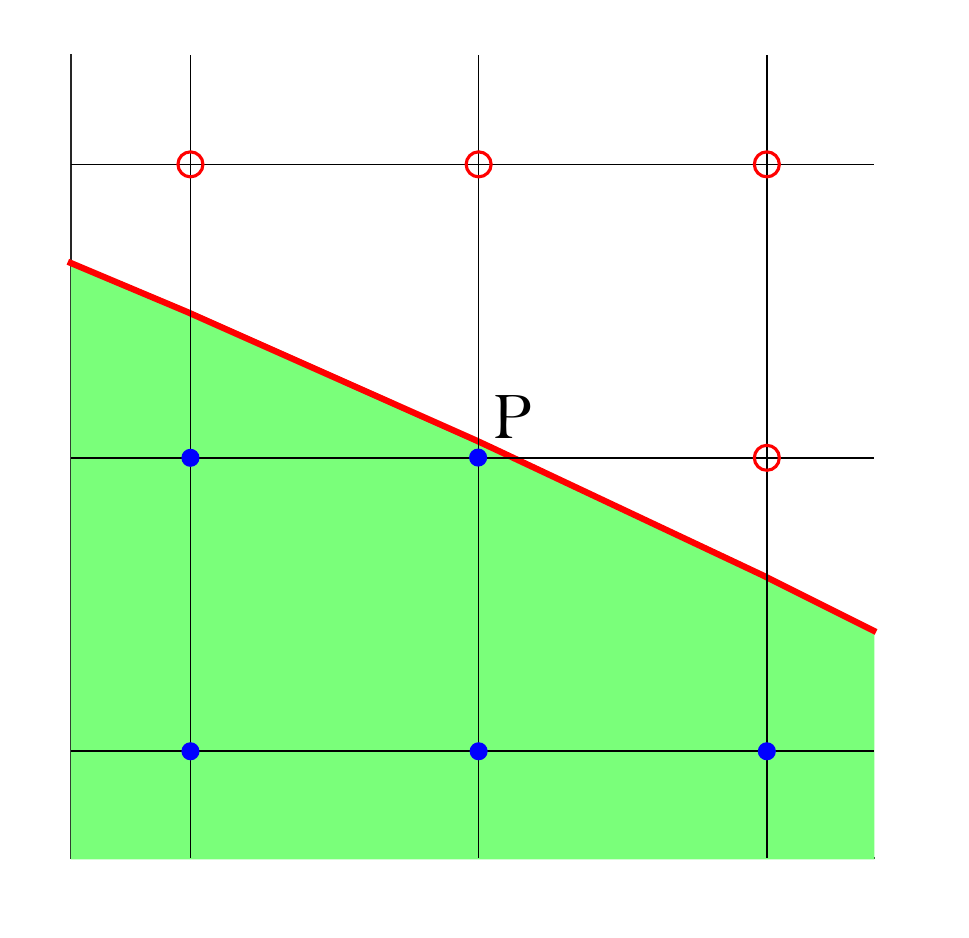}  
\put(-2,50){(a)}
\put(17,40){$\Gamma$}
\end{overpic}
    \end{minipage}
    \begin{minipage}{.4\textwidth}
\begin{overpic}[abs,width=\textwidth,unit=1mm,scale=.25]{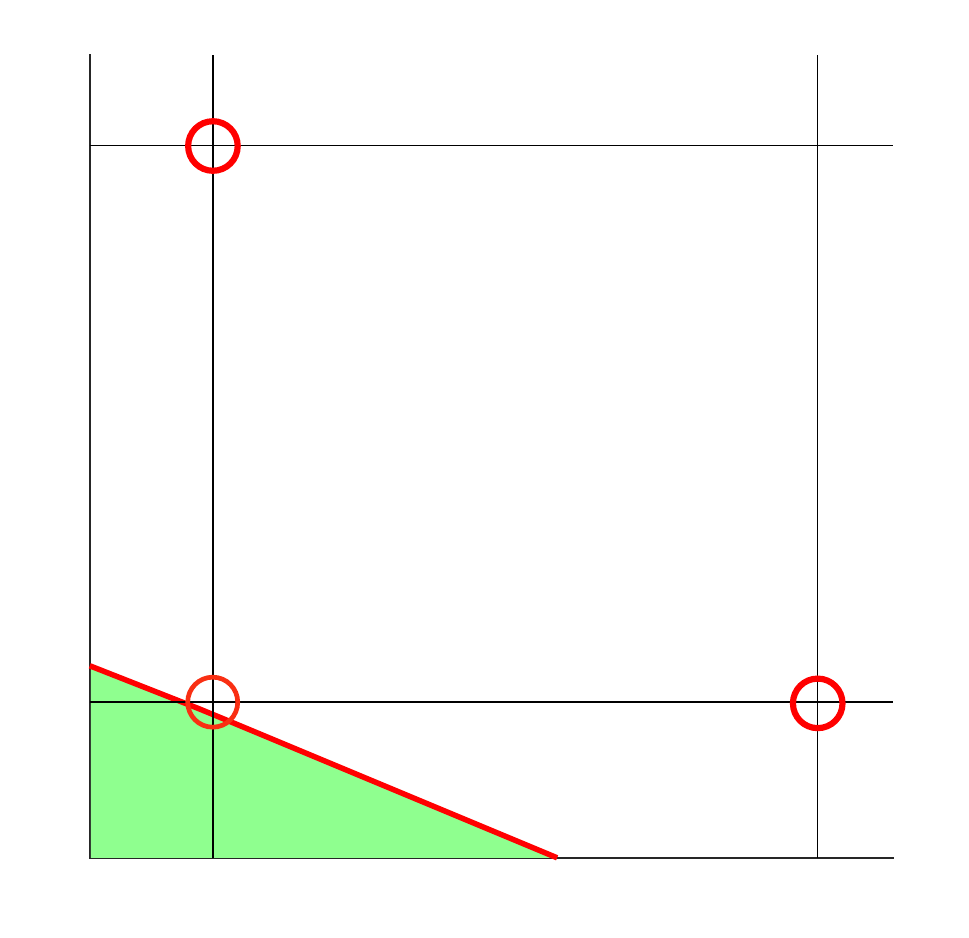}        
\put(-1,50){(b)}
\end{overpic}
    \end{minipage}
\caption{\textit{Grid before and after snapping technique. (a): representation of the cell related to the internal point $P$ (blue points), whose distance from $\Gamma$ is less than $h^2$; (b): zoom-in of the shape of the domain, after the grid point $P$ has changed its classification, from internal to ghost point (red circles).
}}  
\label{fig:snapping}
\end{figure}

\begin{algorithm}
\caption{Computation of the intersection of the boundary with the grid (see Fig. \ref{fig:ref_cell})}\label{alg_ab}
\begin{algorithmic}
\State $k_4 = k_0$
\For{i = 0:3}
\If{$\phi(k_i)\phi(k_{i+1})<0$} 
   \State $\theta = \phi(k_i)/(\phi(k_i)-\phi(k_{i+1})) $ 
   \State $P = \theta k_{i+1}+(1-\theta)k_i$
   \If{$\phi(k_i)<0 $}
       \State ${\bf A}:=P$
       \Else
       \State ${\bf B}:=P$
   \EndIf
\EndIf 
\EndFor
\end{algorithmic}
\end{algorithm}

To compute the integrals showed in Problem~\ref{pro:variational_2}, we use exact quadrature rules. To explain our strategy, let us start considering the product between two test functions $\varphi_i,\varphi_j \in Q_h$ restricted within the cell $K \in \mathcal C$. Since $\varphi_{{j}}({x}_{i},y_i)=\delta_{{i}{j}}, \, i,j \in \mathcal N$,
it is possible to reduce the summations to the vertices of the $K-$th cell (see Fig.~\ref{fig:ref_cell}), yielding 
\begin{align}
\label{eq_int_discr2}
      \left(\varphi_i, \varphi_j\right)_{L^2(K)} = \sum_{\eta,\mu = 0}^{m-1}\left(\varphi_{k_\eta},\varphi_{k_\mu} \right)_{L^2(K)},
\end{align}
where $m$ is the number of edges of the $K-$th cell, and it is $m = 4$ for the internal cells (those cells whose vertices belong to $\Omega_h$), but it can also be $m = 3$ or $m=5$ when a cell intersect the boundary (see, for example, Fig.~\ref{fig:ref_cell} (a)). We observe that the product of two elements in $Q_h$ is an element of $\mathbb{Q}_2(K)$, i.e., the set of bi-quadratic polynomials in $K$.
\begin{figure}[H]
    \centering
\begin{minipage}{.4\textwidth}
\centering
\begin{overpic}[abs,width=0.7\textwidth,unit=1mm,scale=.25]{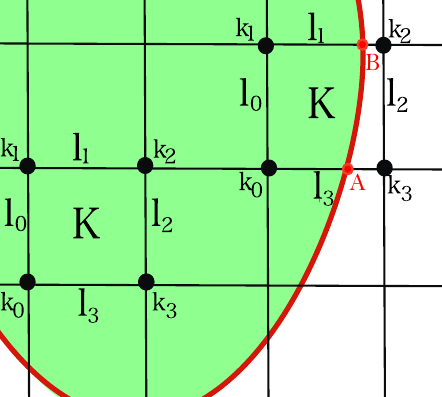}  
\put(-6,39){(a)}
\end{overpic}
    \end{minipage}
\begin{minipage}{.49\textwidth}
\centering \begin{overpic}[abs,width=0.65\textwidth,unit=1mm,scale=.25]{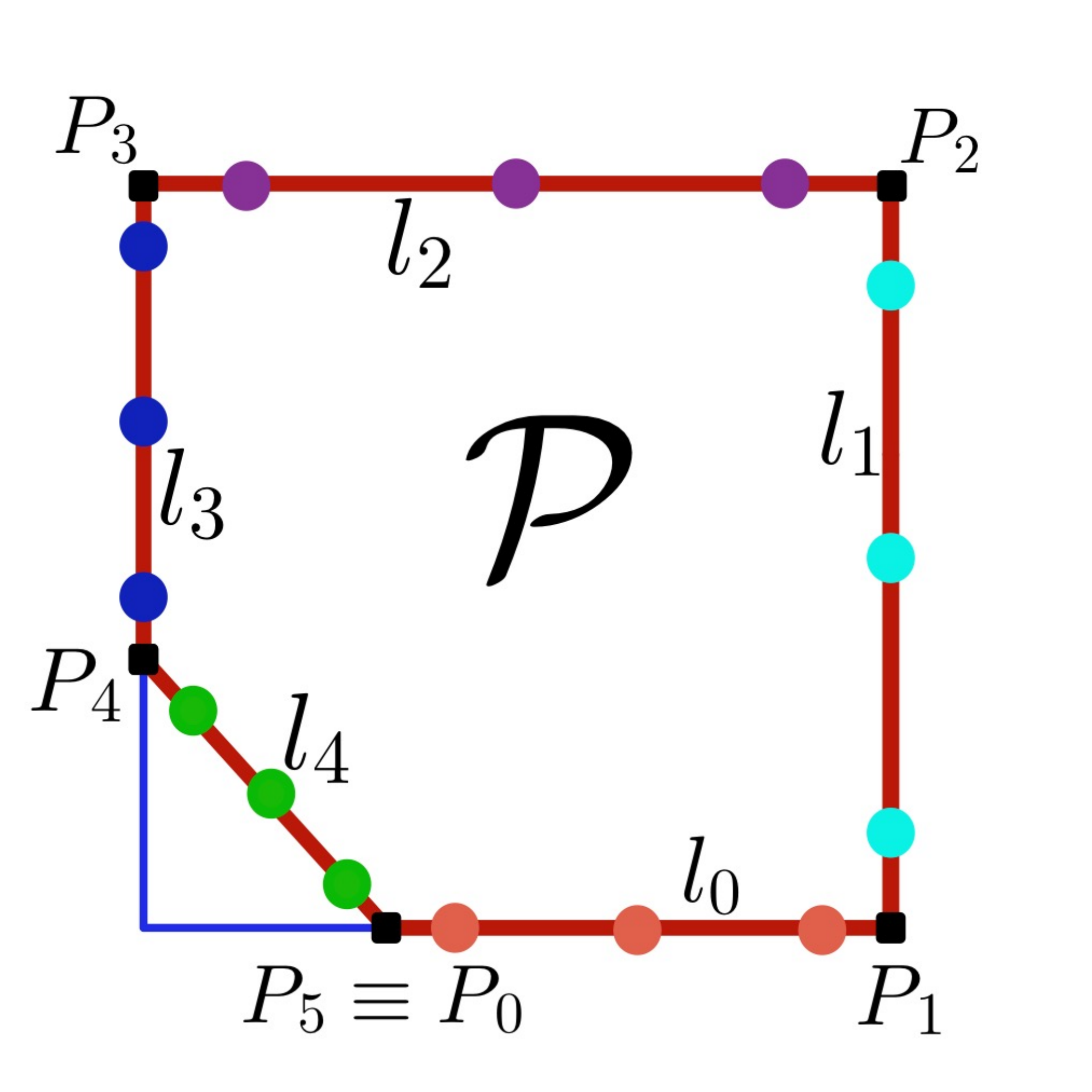}  
\put(-4,44){(b)}
\end{overpic}
    \end{minipage}
\caption{\textit{(a): representation of the indices $\{k_0, \ldots, k_{m-1}\}$, the edges $\{l_0,\ldots,l_{m-1} \}$ of a generic element $K$ and the intersection points $A$ and $B$ from Algorithm\ref{alg_ab}. (b): scheme of the three quadrature points (circles) for each edge $l_i, \, i = 0,\cdots,4$. The squared points represent the vertices $P_i,\, i = 0,\cdots,4,$ of the polygon $\mathcal P$.}}  
\label{fig:ref_cell}
\end{figure}

Let us consider a general integrable function $f$ defined in $\Omega$, with $F(y) = \int f(x,y) dx$ (in our strategy we consider the primitive in $x$ direction; analogue results can be obtained integrating in $y$ direction). We now define a vector $\textbf{F}= (F,0)^\top$ which has the property
$\diver \textbf{F} = f$.
Thus, we have
\begin{align}
\label{eq_diverg}
    \int_K f dx dy = \int_K \diver \textbf{F} \,dx\, dy = \int_{\partial K} \textbf{F}\cdot \textbf{n}\, dl,
\end{align}
where we applied Gauss theorem, and $\textbf{n} = (n_x,n_y)$ is the outer normal vector to $K$.
If $\mathcal{P}$ is the generic polygon with $m$ edges $l_r, \, r = 0,\ldots,m-1$ (see Fig.~\ref{fig:ref_cell}) 
 we can express \eqref{eq_diverg} as
\begin{align}
\label{eq_gauss_interval}
   \int_\mathcal{P} \diver \textbf{F}\, dx\, dy = \int_{\partial \mathcal{P}} \textbf{F}\cdot \textbf{n}\, d l = \sum_{r=0}^{m-1} \int_{l_r} \textbf{F}\cdot \textbf{n}\, dl = \sum_{r=0}^{m-1} \int_{l_r}  F {n}_x\, dl = \sum_{r=0}^{m-1} \int_{l_r}  F dy. 
\end{align}
To evaluate the integral over the generic edge $l_r$, we  choose the three-point Gauss-Legendre quadrature rule, which is exact for polynomials in $\mathbb P_5(\mathbb R)$. Thus, we write
\begin{equation}
     \int_{l_r}  F dy = \sum_{s = 1}^3w_sF(\widehat x_{r,s},\widehat y_{r,s}) (y_{P_{r+1}}-y_{P_r})
\end{equation}
where $w_s$ 
and $\widehat x_{r,s}, \, s = 1,2,3$ are the weights and the nodes, respectively, of the considered quadrature rule, see Fig.~\ref{fig:ref_cell} (b).
Choosing $f=\varphi_{k_\eta}\varphi_{k_\mu}\in \mathbb{Q}_2(K)\subset \mathbb P_4(K)$, and making use of \eqref{eq_gauss_interval}, we write
\begin{equation}
     \left(\varphi_i, \varphi_j\right)_{L^2(K)} = \sum_{r=0}^{m-1} \left(\sum_{s=1}^3 w_s\Phi_{ij}(\widehat x_{r,s},\widehat y_{r,s}) (y_{P_{r+1}}-y_{P_r})\right)|l_r|.
\end{equation}
where $\Phi_{ij} = \int \varphi_i \varphi_j\, dx$, and the formula is exact because $\Phi_{ij} \in \mathbb P_5(K), \, \forall i,j \in \mathcal N$.

While calculating the above mentioned integrals, we notice that stability issues arise when considering cut cells in proximity of the boundary $\Gamma$. The problem derives from the fact that the size of the cuts 
 cannot be controlled and hence can be arbitrarily small. This may result in a loss of coercivity for the bilinear form. 
In Fig.~\ref{fig:snapping}, we show a case in which the stability of the numerical scheme fails. To avoid instability, 
we evaluate the level set function $\phi$ at the vertices of each cell: if the value is {smaller than a threshold} proportional to a power of the length of the cell, i.e.,
{if $0<-\phi<\zeta h^\alpha$, for suitable chosen $\zeta$ and $\alpha$,} we disregard the respective cell {by setting the level set function in that point to a small positive value, as illustrated in Algorithm~\ref{alg_snap}.} In \cite{astuto2024nodal}, the authors present two distinct scenarios starting from a convex domain: one, where the remaining computational domain retains its convex nature, and a second scenario where convexity is lost due to the application of the snapping back to grid technique. For these two different cases, a different {\em a priori\/} analysis was conducted.
Taking advantage of the variational formulation underlying the numerical method, the authors have proven its convergence with second-order accuracy in the ${L}^2$ norm, for general smooth domains and boundary conditions. This approach allows us to preserve the symmetry and the coercivity of the continuous problem in the discrete setting. 
The convergence analysis of the numerical scheme in question is close to that of other classical unfitted finite element methods (FEMs); see, for example, \cite{Burman2015,Lehrenfeld2016}.

\begin{algorithm}[H]
\caption{Snapping back to grid}\label{alg_snap}
\begin{algorithmic}
\For{$k \in \mathcal N$} 
\If{$\phi(k)<0 \quad \& \quad |\phi(k)|<
\zeta h^\alpha$} 
   \State $\phi(k) := {\tt eps} $ 
\EndIf 
\EndFor
\end{algorithmic}
\end{algorithm}

\section{Time discretization}
\label{sec:time_discretization}
We now discuss discretization in time. We consider a final time $T$ and define the time step $\Delta t = T/M,\, M \in \mathbb N,$ denoting the nodes in time by $t^n = n\Delta t$ and $\mathbb{C}_h^n \approx \mathbb{C}_h(t^n), \, n = 0,\cdots,M$. We {adopt}  a first order semi-implicit approximation of the time derivative of $\mathbb{C}_h$ because our focus is on the computation of the steady-state solution. Making use of Implicit Euler scheme, we have
\begin{equation*}
    \frac{\partial \mathbb{C}_h(t_{n+1})}{\partial t}\approx\frac{\mathbb{C}_h^{n+1} -\mathbb{C}_h^{n}}{\Delta t}
\end{equation*}
so that \eqref{eq:C_disc} can be rewritten as 
\begin{align} \notag
    \left(\, \mathbb{C}_h^{n+1}, \mathbb{B}_h \,\right) + \Delta t\, \widetilde D^2 \left(\, \nabla \mathbb{C}_h^{n+1}, \nabla \mathbb{B}_h \,\right) &- \Delta t\,   \left(\, \Phi''(p_h^n)\nabla p^n_h \otimes \nabla p^n_h + \frac{\nabla \sigma^n_h \otimes \nabla p^n_h + \nabla p^n_h \otimes \nabla \sigma^n_h}{2}, \mathbb{B}_h\,\right)\\ \label{eq:C_time_dep}  &+ \Delta t\,  \widetilde \nu \left(\, \left(|| \mathbb{C}_h^{n} || + \varepsilon \right)^{\gamma - 2} \mathbb{C}_h^{n+1},\mathbb{B}_h \right)
    = \left(\, \mathbb{C}_h^{n}, \mathbb{B}_h \,\right),
\end{align}
where we add a \textit{regularization parameter} $\varepsilon > 0$, to prevent the instability coming from the division by zero {in the metabolic term}, when $\gamma < 1$ (see \cite{astuto2022comparison, astuto2023asymmetry, astuto2023finite} for more details). 
In order to solve Problem~\ref{pro:variational_2}, we adopt a semi-implicit numerical scheme. At the time step $t^n$ we compute the following
\begin{subequations}
\begin{align}
\label{eq:p_time_dep}
  \left(\, \nabla q_h, \left(\mathbb{C}^n_h+r\mathbb{I}_h\right) \nabla p_h^n  \,\right)  & = \left(\, S_h , q_h \right) & \\ \label{eq:sigma_time_dep}
  \left(\, \nabla v_h, \left(\mathbb{C}^n_h+r\mathbb{I}_h\right) \nabla \sigma_h^n \,\right)  & = \left(\, \Phi'''(p^n_h)\nabla p^n_h\cdot\left(\mathbb{C}_h^n + r\mathbb I_h\right)\nabla p^n_h, v_h \right). &
\end{align} 
\end{subequations}
Also the nonlinear terms are calculated explicitly, such as 
$$\left(\, \Phi''(p_h^n)\nabla p^n_h \otimes \nabla p^n_h + \frac{\nabla \sigma^n_h \otimes \nabla p^n_h + \nabla p^n_h \otimes \nabla \sigma^n_h}{2}, \mathbb{B}_h\,\right),$$ as well as the nonlinear metabolic term $$\left(\, \left(|| \mathbb{C}_h^{n} || + \varepsilon\right)^{\gamma - 2} \mathbb{C}_h^{n+1},\mathbb{B}_h \right).$$

\section{Numerical results}
\label{sec:numerical_results}
In this section we apply the ghost nodal FEM \cite{astuto2024nodal} with the dual purpose of testing the newly developed method by verifying its accuracy and independence on grid orientation, and exploring the behavior of the solution of the biological transportation  model, for different geometries and various choices of the entropy generator function $\Phi(p)$. 
By specifying $\Phi(p)$, which affects the entropy dissipation term, 
we model various phenomena, as highlighted in \cite{portaro2022emergence}. 
A common example is given by $\Phi(p) = p^2 / 2$, {which corresponds to the minimization of Joule heating \cite{portaro2022emergence, khan2018entropy}}.
In this case, the pumping energy is determined by Joule's law, which states that the system power density is given by the scalar product of the current $J = \C \grad p$ and the potential gradient $\grad p$. {This scenario has been extensively studied in the literature, e.g.\ \cite{astuto2022comparison, astuto2023asymmetry, marko_pilli}}.
  
Another interesting entropy generator is given by {the Boltzmann entropy generator} $\Phi(p) = (p+1) (\ln (p+1) - 1)$, which converts the kinetic part of the energy \eqref{eq:energy_scaled} into the Fisher information \cite{portaro2022emergence}. {This quantity plays an important role in various mathematical fields, ranging from information theory and statistical parameter estimation \cite{rao1992information}, to statistical mechanics and fluctuations in thermodynamic equilibrium \cite{crooks2011fisher}, and even mesoscopic models in biological network formation \cite{portaro2024measure}.} 
{In \cite[Lemma 2.6]{arnold2001convex}, the sub- and super-entropies used in the relative entropy argument leading to a log-Sobolev inequality were the Boltzmann entropy and the quadratic entropy, respectively. In our context, there are no restrictions on the choice of entropy. Therefore, we can use a quartic entropy generating function $\Phi(p) = {p^4}/{12}$ to test the numerical stability of our scheme and explore beyond the quadratic threshold previously mentioned. This choice is noteworthy because nonlinear and non-local PDEs often exhibit poor behavior with higher powers, such as those seen in quartic entropy generating functions. For similar numerical reasons, a linear combination of the Fisher entropy and quartic entropy generators is also of interest.}

In summary, the expressions we choose for the second derivative of the entropy generating function $\Phi(p)$ are:
\begin{description}
    \item[\hspace{4cm}quartic] $\; \; \quad \qquad \Phi''(p) = p^2$
    \item[\hspace{4cm}Fisher] $\qquad \qquad \Phi''(p) = (p+1)^{-1}$ 
    \item[\hspace{4cm}combination] $\quad \Phi''(p) = 0.5\,p^2 + 0.5\,(p+1)^{-1}$
\end{description}

Regarding the domain $\Omega$, we start with a circular domain, to see if there are effects coming from the anisotropy 
of the grid. The equations presented in (\ref{eq:C_time_dep}-\ref{eq:sigma_time_dep}) are symmetric with respect to the $x$ and $y$ axis. Therefore, we expect that a symmetric domain would maintain their properties. {In Fig.~\ref{fig:symm},  the time evolution of the solution at selected times is presented to illustrate the moment in which symmetry is lost. Our numerical evidences suggest that the loss of symmetry is entirely attributable to the metabolic term. There is no loss of symmetry for $\widetilde \nu = 0$ and $\widetilde \nu \neq 0, \, \gamma > 1$. These tests led us to suspect that the loss of symmetry is mainly due to the small denominator in the term $||\mathbb C + \varepsilon||^{\gamma - 2}$, when $\gamma < 1 $ and $\mathbb C \approx 0$. For this reason, we conducted other tests that confirmed our theory. We also observe no loss of symmetry in the following tests: $\widetilde \nu \neq 0, \, \gamma < 1$ and $\varepsilon$ big enough (say $\varepsilon > 10^{-3}$) and $\widetilde \nu = \varepsilon^{2-\gamma}$, with $\varepsilon$ small (e.g., $\varepsilon = 10^{-4}$).}

{A second step to demonstrate that the considered numerical scheme is grid-independent is to examine the behavior of the solution when the domain $\Omega$ is rotated. This test is not meaningful in a circular domain. Therefore, we select a leaf-shaped domain to illustrate this point, as our focus is on the application of biological network formation in plant leaves.}

Here we define the different domains that we consider.  In all our numerical results we set {\tt eps} in Algorithm \ref{alg_snap} to Matlab machine epsilon.

\subsubsection*{Circular domain} Let us start with a circle centered at $(x_0,y_0) \in \Omega \subset R = [0,1]^2$. The level-set function is 
\[\phi=\sqrt{(x-x_0)^2+(y-y_0)^2}-r_c,\]
where $r_c$ is the radius of the circle. In our numerical tests, $x_0 = y_0 = 0.5$ and $r_c = 0.45$.

\begin{figure}[H]
    \begin{tabular}{|c|c|c|c|}
    \hline &
     N = 200 & N = 400 & N = 800 \\
\hline \begin{turn}{90}quartic \end{turn} &
     \includegraphics[width=0.3\textwidth]{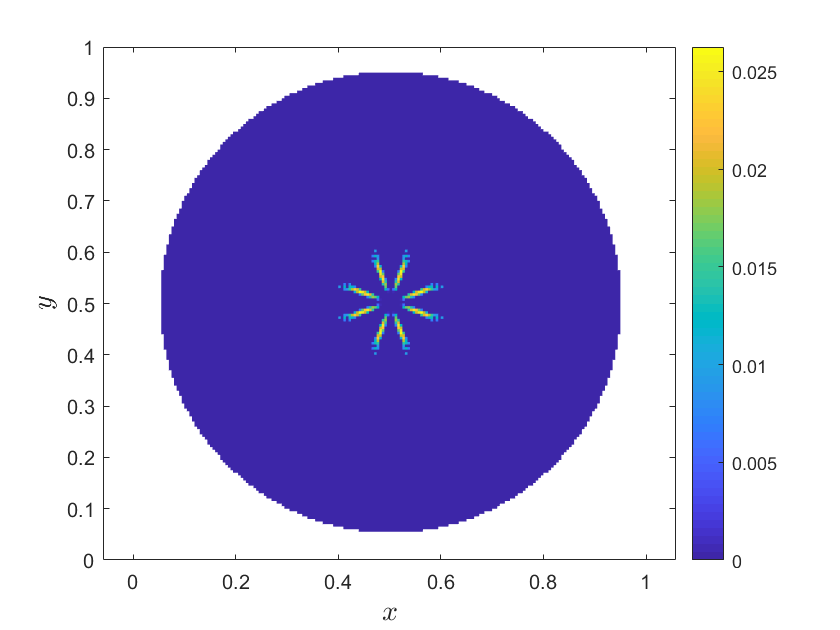} & \includegraphics[width=0.3\textwidth]{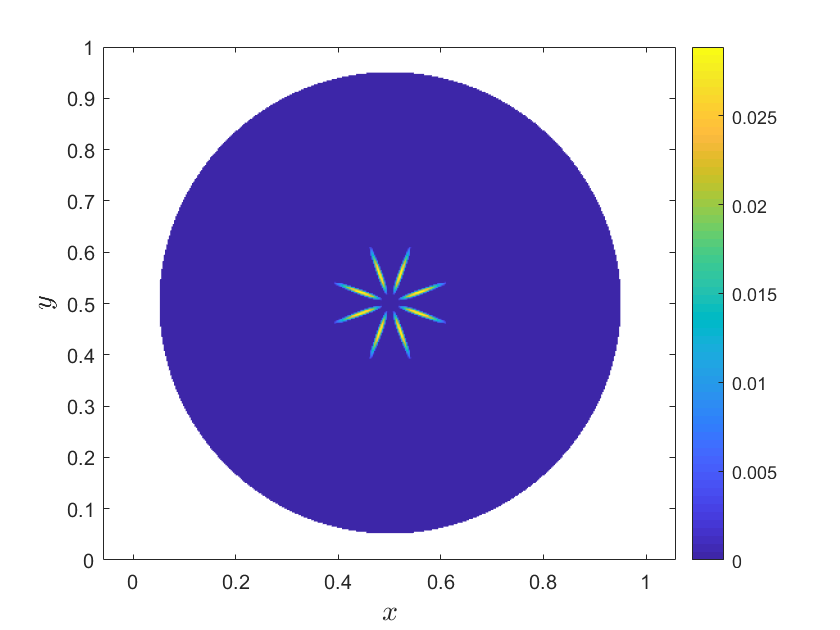}  &  
     \includegraphics[width=0.3\textwidth]{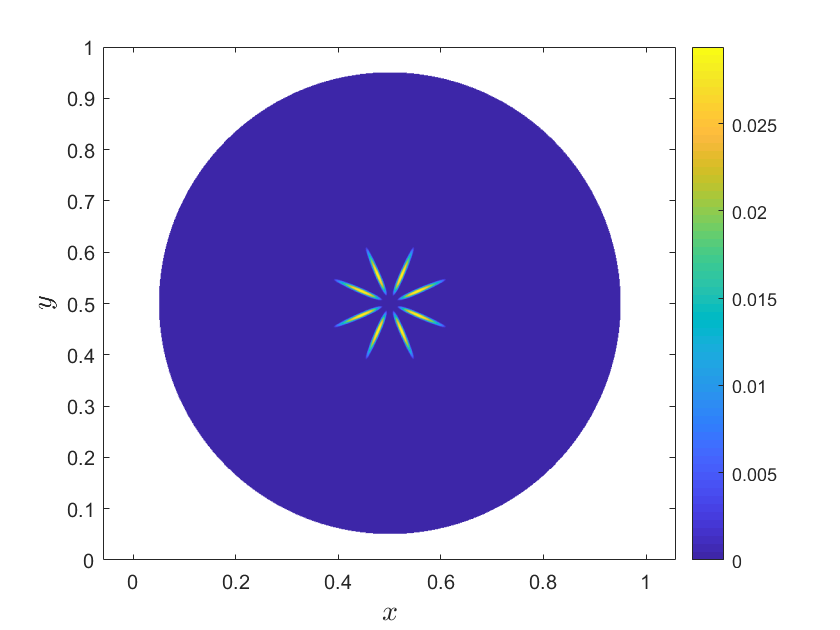} 
     \\
     \hline \begin{turn}{90}Fisher \end{turn} &
     \includegraphics[width=0.3\textwidth]{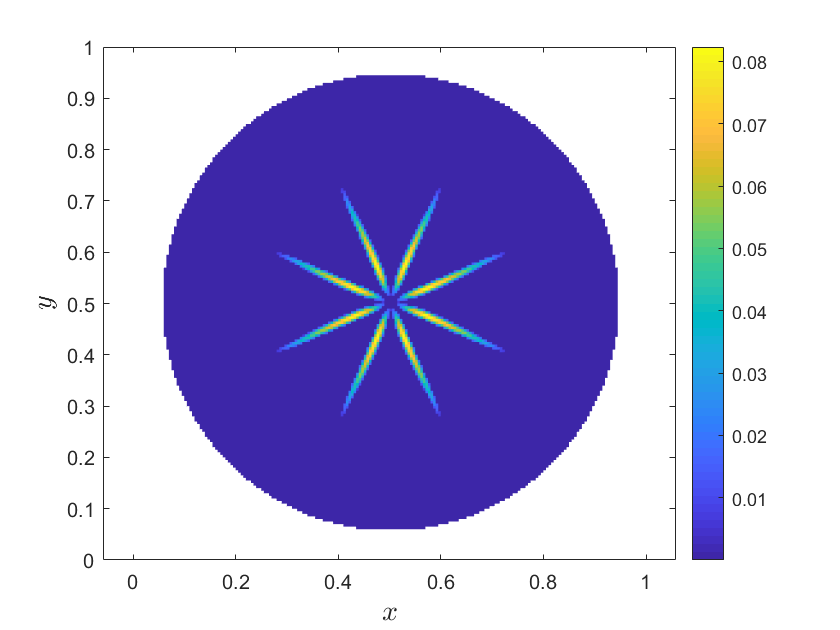} & \includegraphics[width=0.3\textwidth]{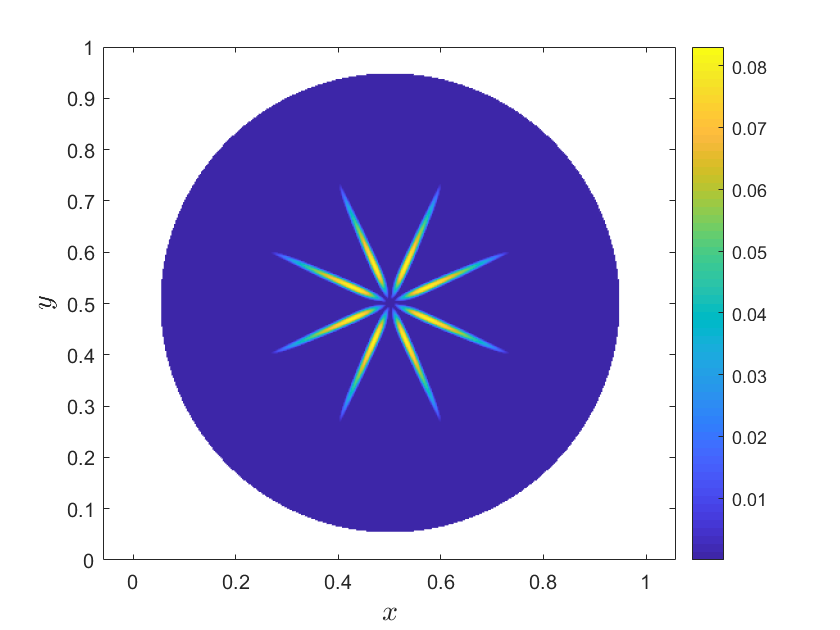}  &  \includegraphics[width=0.3\textwidth]{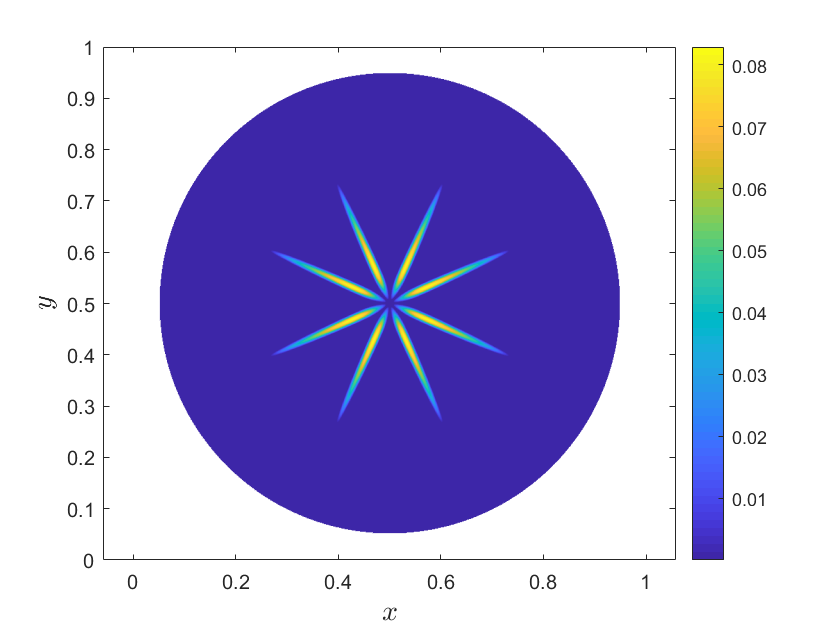}  \\ \hline
    \end{tabular}
    \caption{\textit{Configuration of the solution in a circular domain, changing the number $N$ of cells. All the parameters are in Table~\ref{tab:parameters}, and for the  Fisher case we choose $\widetilde D = 2.5\cdot 10^{-5}, \widetilde \nu = 0.1$. In all cases, the stable radially symmetric solution has eight main branches, i.e. it is symmetric with respect to rotations by $k\pi/4$, $k\in\mathbb{Z}$.}}
    \label{fig:configuration}
\end{figure}

\subsubsection*{Leaf-shaped domain} Here we present a leaf-shaped domain $\Omega \subset R = [0,1]^2$. In this case, the level-set function is the intersection of two circles centered at $C_1\equiv(x_1,y_1)\in \Omega$ and $C_2\equiv(x_2,y_2)\in \Omega$, with radius $r_c$. 
The respective level-set functions of the two circles are
\[\phi_1 = \sqrt{(x-x_1)^2+(y-y_1)^2}-r_c, \quad \phi_2 =\sqrt{(x-x_2)^2+(y-y_2)^2} -r_c,\]
and the one of the leaf-shaped domain is
\[ \phi = \max\{\phi_1,\phi_2\}. \]
In our numerical tests $x_1 = 0.4, x_2 = 0.6, y_1 = y_2 = 0.5$ and $r_c = 0.4$.

\subsubsection*{Rotated leaf-shaped domain}
Now we consider a domain which is rotated by an angle $\theta$ counterclockwise. 
The corresponding level-set function
is obtained by rotating the two centers $C_1$ and $C_2$. The new centers are then defined as follows 
\begin{align*}
&(x_1, y_1) \in \Omega, 
\quad  \widetilde x_1 =  x_1\cos(\theta) -  y_1\sin(\theta), \, \widetilde y_1 =  x_1\sin(\theta) +  y_1\cos(\theta) \\
&(x_2, y_2) \in \Omega, 
\quad \widetilde x_2 =  x_2\cos(\theta) -  y_2\sin(\theta), \, \widetilde y_2 =  x_2\sin(\theta) + y_2\cos(\theta) \\ 
&R_1 = \sqrt{(x-\widetilde x_1)^2+(y-\widetilde y_1)^2}, \quad R_2 = \sqrt{(x-\widetilde x_2)^2+(y-\widetilde y_2)^2} \\         
&\phi_1 = R_1(x,y)-r_c, \quad \phi_2 = R_2(x,y)-r_c, \quad \phi = \max\{\phi_1,\phi_2\},
\end{align*}
where $r_c$ is the radius of the circles.

The choice of parameters is showed in Table \ref{tab:parameters}, while the expression for the source term $S$ is the following:
\[
S = \exp\left({-\omega((x-x_S)^2 + (y-y_S)^2)}\right),
\]
with $x_S = y_S = 0.5$ for the circular domain, and $x_S = 0.5, y_S = 0.2$ for the leaf-shaped domain. 
\begin{table}[H]
    \centering
    \begin{tabular}{||c||c||c||c||c||c||c||c||}
    \hline \hline
       $\widetilde D$ & $\widetilde \nu$ & $\varepsilon$ & $\gamma$ & $r$ & $\omega$  & $T$ & $\Delta t $ \\ \hline \hline
    $4\cdot 10^{-6}$ & $4\cdot 10^{-2}$ & $10^{-4}$ & $0.75 $ & $5\cdot 10^{-3}$ & 500 & 400 & $h$\\
  \hline \hline
    \end{tabular}
    \caption{\textit{Parameters of the numerical tests.}}
    \label{tab:parameters}
\end{table}

\begin{figure}[H]
    \begin{tabular}{|c|c|c|}
    \hline t = 6.25 & t = 12.5 & t = 18.75 \\
\hline
  \includegraphics[width=0.3\textwidth]{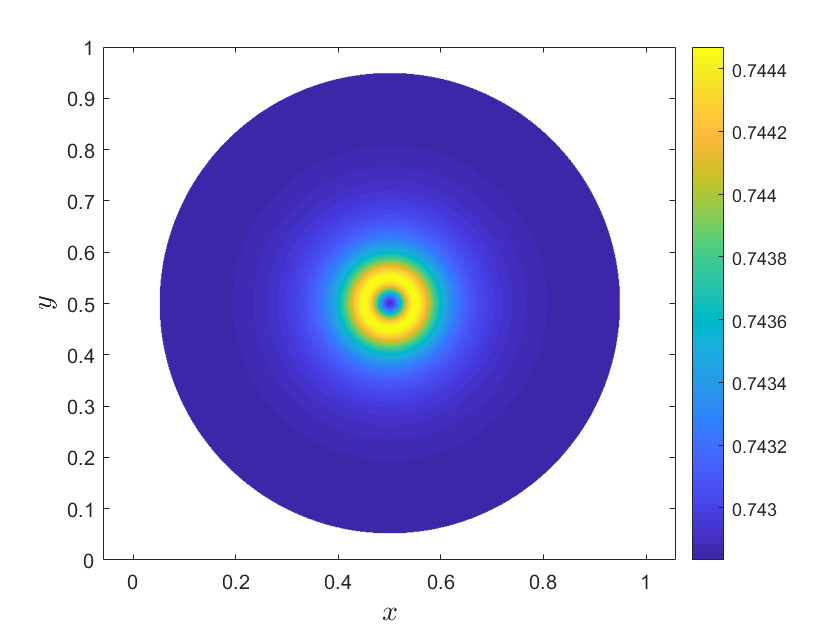}  &  \includegraphics[width=0.3\textwidth]{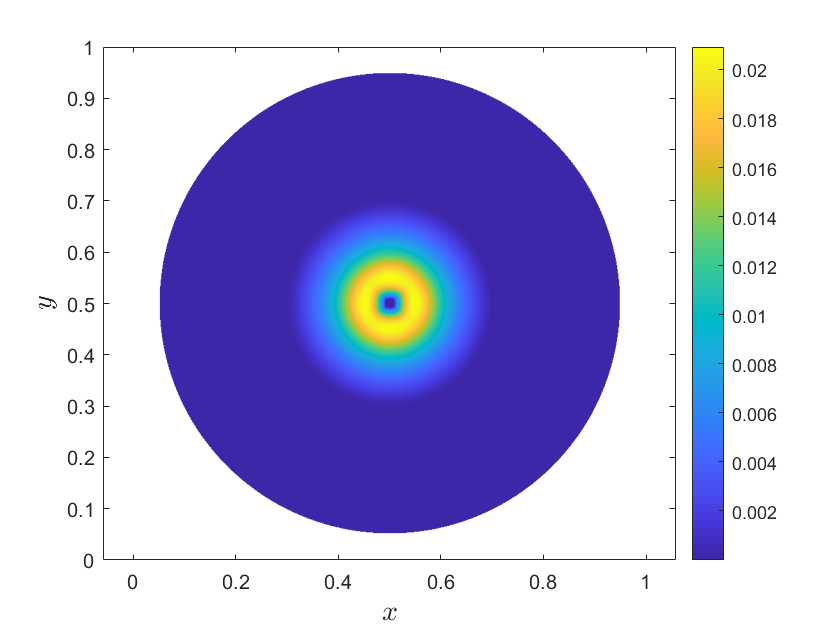} &
\includegraphics[width=0.3\textwidth]{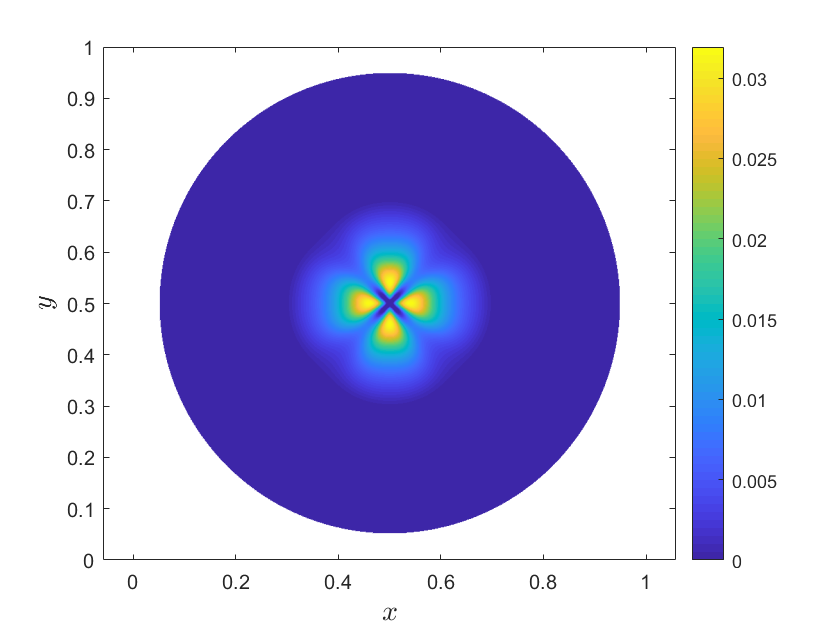} \\  \hline  t = 25 & t = 37.5 & t = 400 \\
\hline
\includegraphics[width=0.3\textwidth]{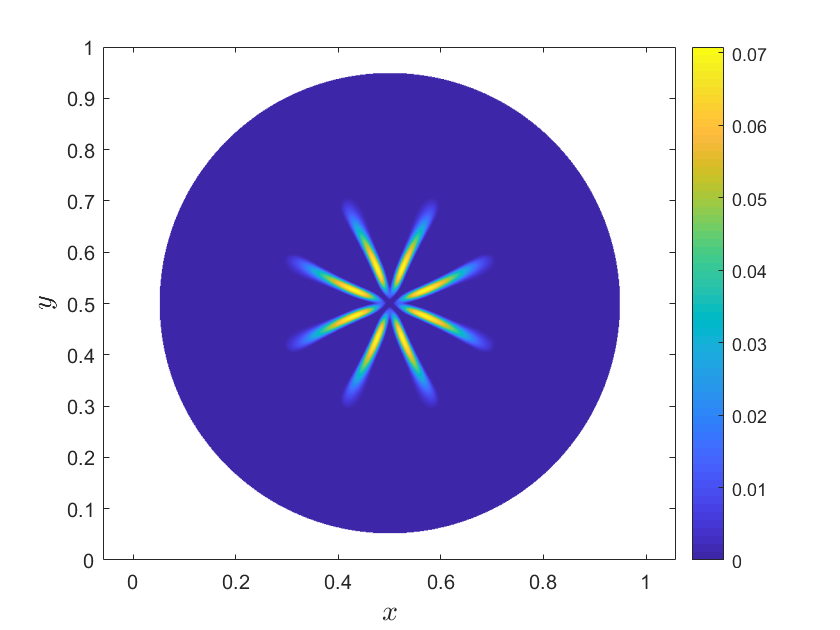}  &  \includegraphics[width=0.3\textwidth]{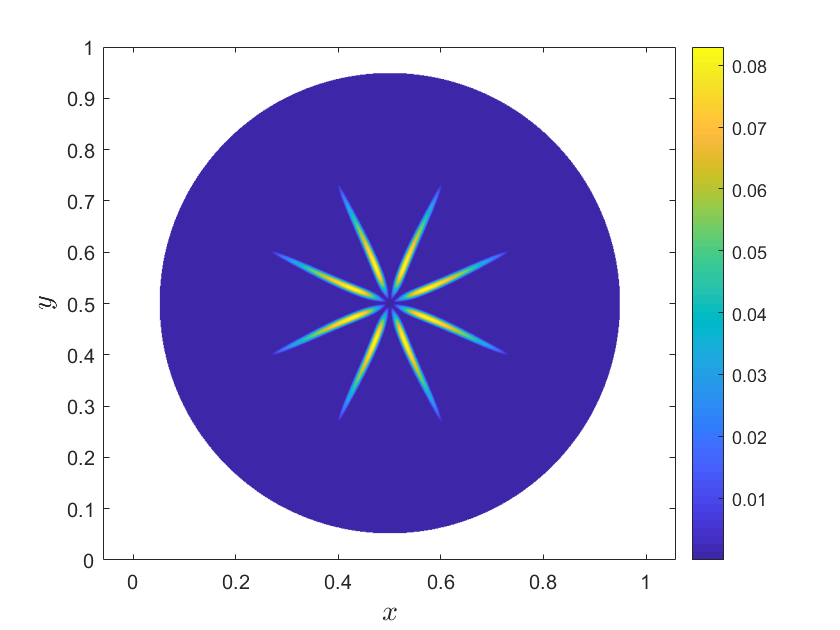} &
 \includegraphics[width=0.3\textwidth]{Figures_2D/circular_Fisher_N800.png} \\
\hline
 \end{tabular}
    \caption{\textit{Snapshots of the solution at various times, illustrating the progressive loss of symmetry. In these plots $N = 800$ and the other parameters are in Table \ref{tab:parameters}, with $\widetilde D = 2.5\cdot 10^{-5}$ and $\widetilde \nu = 0.1$.}}
    \label{fig:symm}
\end{figure}

The results are {illustrated after} the steady state is reached. To be sure of this, we show the energy decay of the entropy functional in \eqref{eq:energy_scaled}. Our numerical results primarily come from solving the system for the conductivity in the scalar case. Additional details regarding the tensor case are illustrated in Fig.~\ref{fig:entropy}. In Fig.~\ref{fig:accuracy}, we show the accuracy of the numerical solutions, and it is important to remark that the numerical solution reaches its configuration even in the coarsest grid, as shown in Fig.~\ref{fig:configuration}. In absence of the exact solution, we apply Richardson extrapolation technique (see, e.g.,~\cite{richardson1911ix}) to estimate the order of the method. To calculate the error, we use the Wasserstein distance\footnote{{To calculate the Wasserstein distances between two vectors, we make use of the matlab function available at https://github.com/nklb/wasserstein-distance.}}  \cite{carrillo,otto1996double,ottoF} between solutions computed with different spatial discretizations. The study in \cite{astuto2023asymmetry} demonstrates the benefits of using the Wasserstein distance over traditional $L^p$ norms for calculating the accuracy order, particularly in scenarios where the solution features multiple branches that vary in both position and value. In Fig.~\ref{fig:rot_domain_F}, we illustrate the results by rotating the leaf-shaped domain by an angle $\theta = \pi/4$. Table \ref{tab:rotation} further confirms these results, showing that the agreement improves as the grid is refined. In our numerical tests, we observed that the stationary solution does not depend separately on the three parameters $D^2, c^2$, and $\nu$, but rather on their ratio, as illustrated in \eqref{eq:energy_scaled}, and for this reason, we will only consider the variation of the parameters $\widetilde D$ and $\widetilde \nu$. As said before, we remark that the time $t$ considered here is actually a rescaled variable, obtained by multiplying it by the constant $c^2$. Moreover, {as expected,} the steady state is  reached {after a longer time for smaller values of the parameter $ \widetilde \nu$}, as shown in Fig.~\ref{fig:SS_nu}. In Fig.~\ref{fig:Dtilde}, we observe that as we decrease the parameter $\widetilde D$, there is a noticeable increase in the number of branches depicted.

\begin{figure}[H]
    \begin{tabular}{|c|c|}
\hline  quartic & Fisher \\ \hline \includegraphics[width=0.4\textwidth]{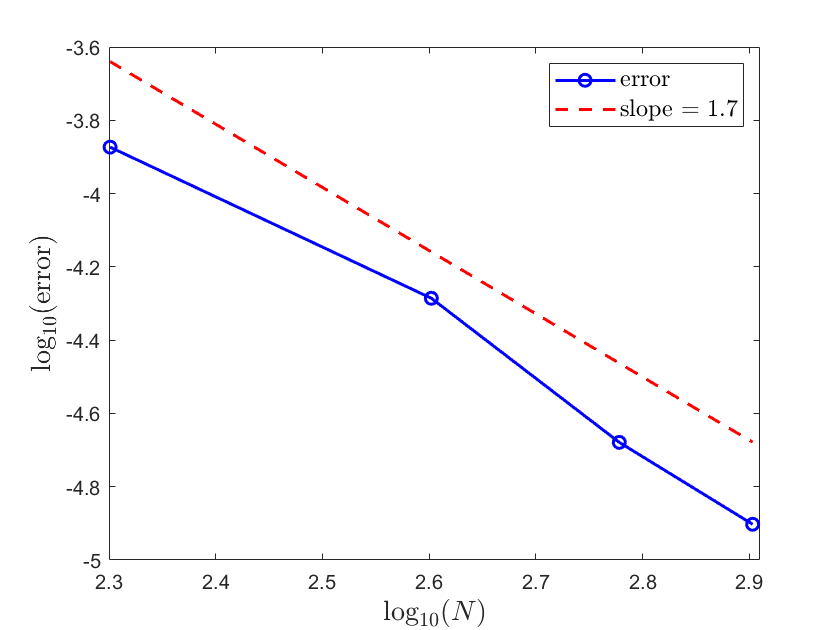} & \includegraphics[width=0.4\textwidth]{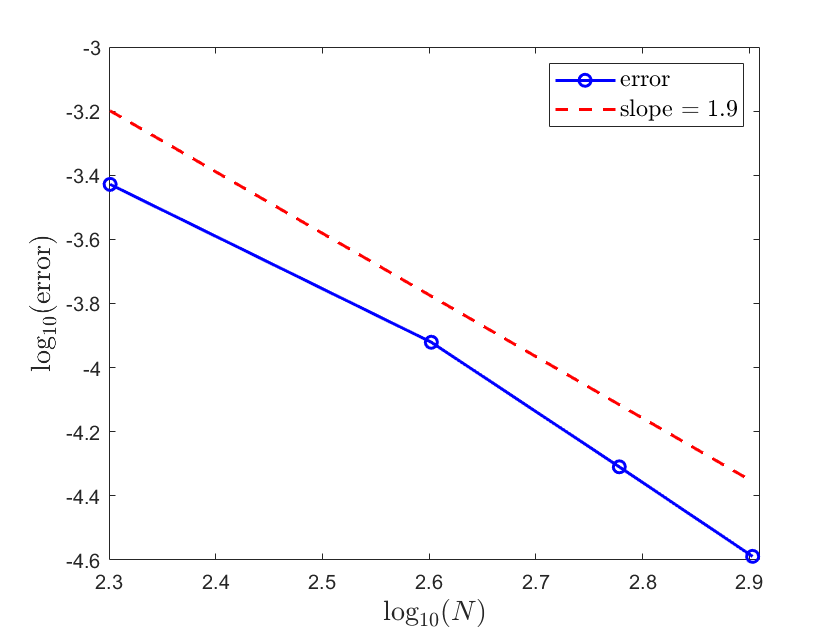} \\ \hline
    \end{tabular}
\caption{\textit{Error analysis for quartic and Fisher cases in a circular domain. All the parameters are in Table~\ref{tab:parameters}, and for the  Fisher case we choose $\widetilde D = 2.5\cdot 10^{-5}, \widetilde \nu = 0.1$.}}
    \label{fig:accuracy}
\end{figure}

In Fig.~\ref{fig:entropy}, we illustrate the results {for various} convex entropy functions, for the scalar and tensor cases. The results are in agreement, and in some cases the tensor show more branches. 
Furthermore, after showing that for $r = 5\cdot 10^{-3}$ we reach an accurate solution for $N\geq 800$, in agreement with the results shown in Figs.~\ref{fig:accuracy} and \ref{fig:log_r5em3}, in Figs.~\ref{fig:log_r1em3}-\ref{fig:log_r5em4} we show different tests for smaller values of $r$, the parameter that describes the isotropic background permeability of the medium. In \cite{astuto2023asymmetry}, the relation between the parameter $r$ and the symmetry of the solution is illustrated. When $r$ decreases, the network becomes more articulated, and the symmetry of the solution deteriorates. Anyway, thinner branches appear in the domain, and the typical plot that we used so far is not suitable to capture them. For this reason, in Fig.~\ref{fig:log_r1em3}-\ref{fig:log_r5em4} we {produce a contour plot of the} logarithm  of the solution, which emphasizes low values of $\mathbb{C}$, and its zoom-in of a part of the domain. In the Poisson equations for $p$ and $\sigma$, the term $ r\mathbb I$ tends to keep the conductivity away from zero.
When $\mathbb C \approx r$,  the conductivity becomes almost isotropic and the branches disappear. For this reason, in Fig.~\ref{fig:contour_r}, we show also the contour line for $\mathbb C = r$ (green, and $\mathbb C = r/2$ (red)), that corresponds to the smallest branches in the $\log \|\mathbb C\| $ plot.
We show two different cases, with $r = 10^{-3}$ in Fig.~\ref{fig:log_r1em3}, and $r = 5\cdot 10^{-4}$ in Fig.~\ref{fig:log_r5em4}. 
We observe that the smaller the value of $r$, the more articulated the solution, with more and more small branches that occupy a wider region of the computational domain. When  $\mathbb{C}$ is smaller than $r$ then isotropic permeability dominates, and the branches disappear. On the other hand, for smaller and smaller values of $r$, smaller scales appear in the solution. Because of such small scales, to get accurate solutions become harder and harder as $r$ decreases, since more and more grid points are necessary to resolve the fine details of the solution.


\begin{figure}[h]
    \begin{tabular}{|c|c|c|}
\hline \begin{turn}{90}$\qquad E[\mathbb C]$\end{turn}
      &  \includegraphics[width=0.4\textwidth]{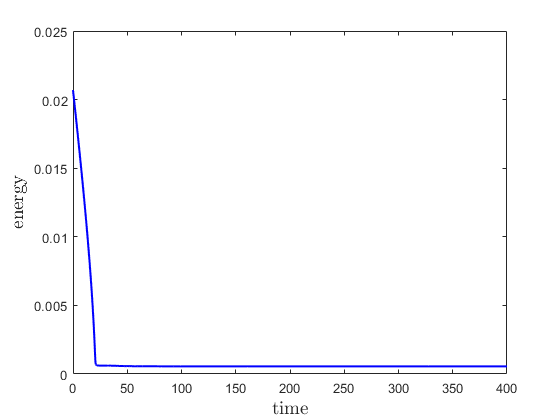} & \includegraphics[width=0.4\textwidth]{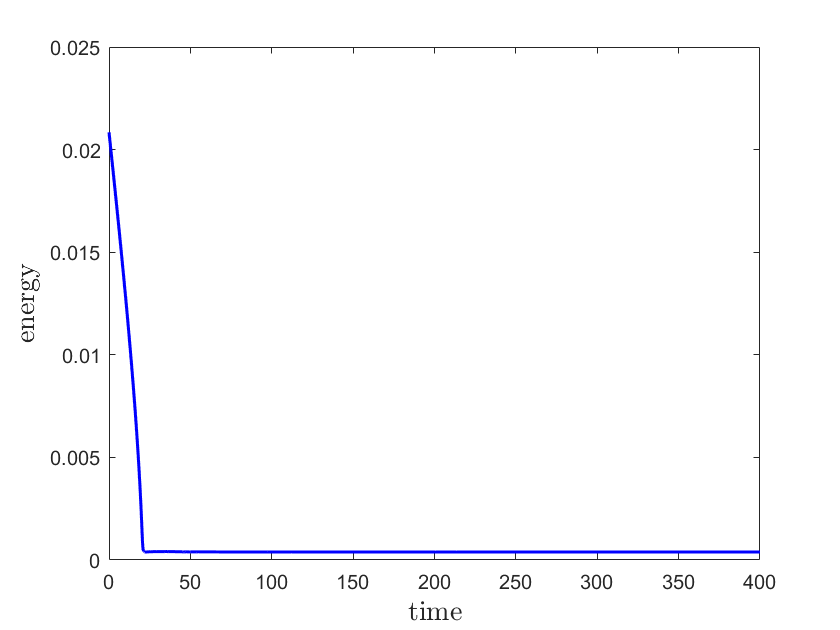}\\ \hline
    \begin{turn}{90} $\qquad \mathbb C$ \end{turn}
& \includegraphics[width=0.4\textwidth]{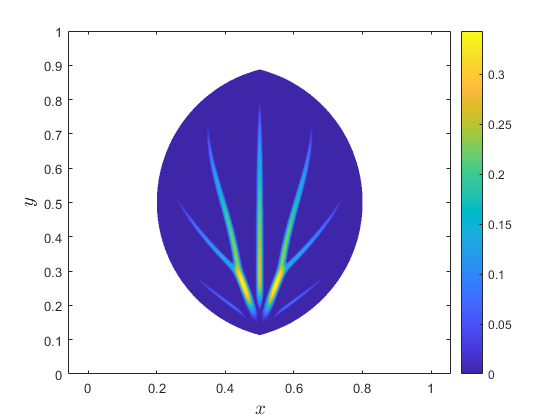}  &  \includegraphics[width=0.4\textwidth]{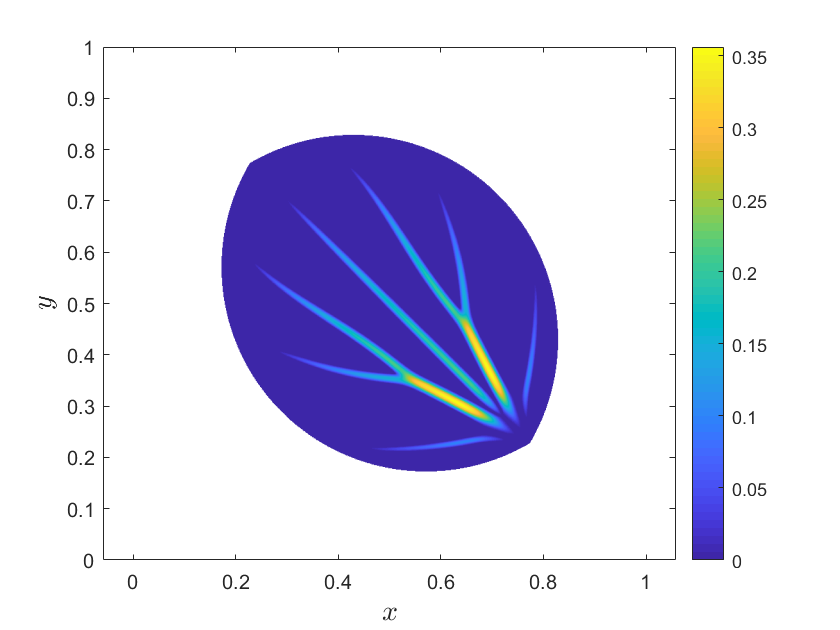}  \\ \hline
    \end{tabular}
    \caption{\textit{Comparison of solutions and energy behavior in leaf-shaped domains, with and without rotation.  In these plots $N = 800$ and parameters in Table \ref{tab:parameters}.
    }}
    \label{fig:rot_domain_F}
\end{figure}

\begin{table}[]
    \centering
    \begin{tabular}{||c||c||}
    \hline \hline
N  & $\mathcal W$   \\ \hline \hline
  400 & 7.3178e-04 \\ \hline \hline
 600 & 6.1846e-04 \\ \hline \hline
  800 & 3.1681e-04 \\  \hline \hline
    \end{tabular}
    \caption{\textit{Wasserstein distance
    ${\mathcal W}$ between the solutions obtained with and without rotation, for different number of cells. }}
    \label{tab:rotation}
\end{table}

\begin{figure}[h]
    \begin{tabular}{|c|c|c|c|}
    \hline & $\widetilde \nu = 4\cdot 10^{-2}$ & $\widetilde \nu = 2\cdot 10^{-2}$  & $\widetilde \nu = 10^{-2}$ \\
\hline \begin{turn}{90}$\qquad E[\mathbb C]$\end{turn}
      &  \includegraphics[width=0.3\textwidth]{Figures_2D/energy_base_case_N800.png} & \includegraphics[width=0.3\textwidth]{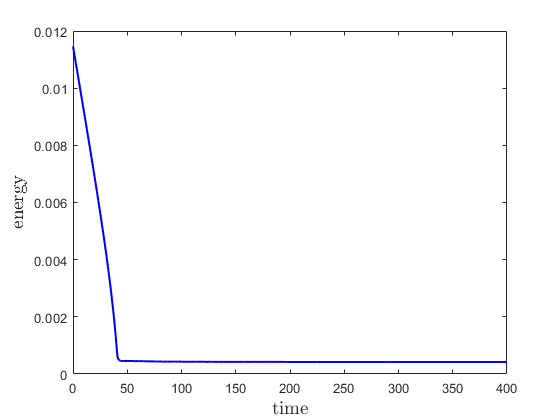} & \includegraphics[width=0.3\textwidth]{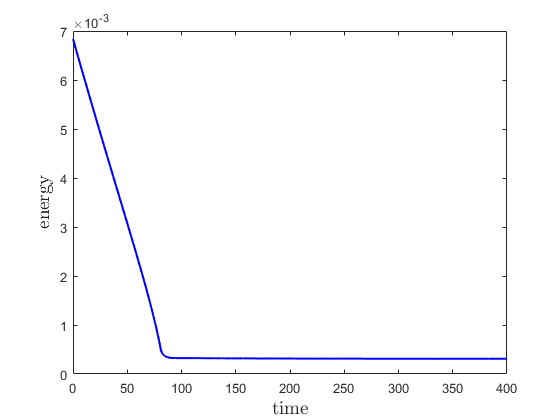}\\ \hline
    \begin{turn}{90} $\qquad \mathbb C$ \end{turn}
& \includegraphics[width=0.3\textwidth]{Figures_2D/solu_base_case_N800.png}  &  \includegraphics[width=0.3\textwidth]{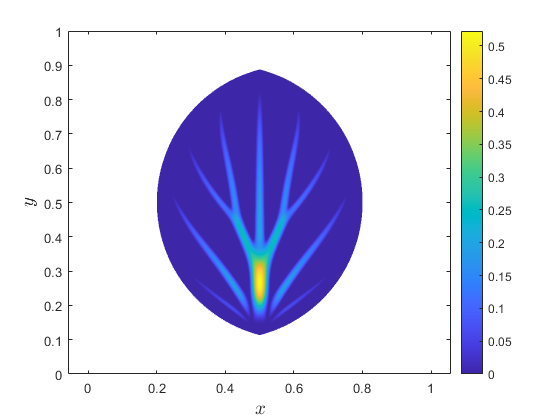}  &  \includegraphics[width=0.3\textwidth]{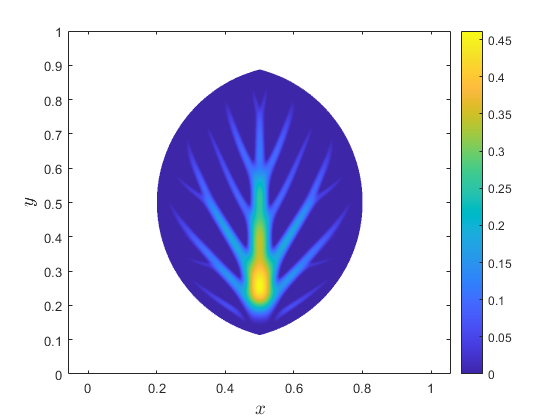}  \\ \hline
    \end{tabular}
\caption{\textit{Comparison of the solutions and of the energy behaviour changing the metabolic coefficient $\widetilde \nu$, with leaf-shaped domain.  In these plots $N = 800$ and other parameters are in Table \ref{tab:parameters}.}}
\label{fig:SS_nu}
\end{figure}

\begin{figure}[h]
    \begin{tabular}{|c|c|c|} \hline
     $\widetilde D = 4\cdot 10^{-6}$ & $\widetilde D = 2\cdot 10^{-6}$  & $\widetilde D = 10^{-6}$ \\
\hline \includegraphics[width=0.3\textwidth]{Figures_2D/solu_base_case_N800.png}  &  \includegraphics[width=0.3\textwidth]{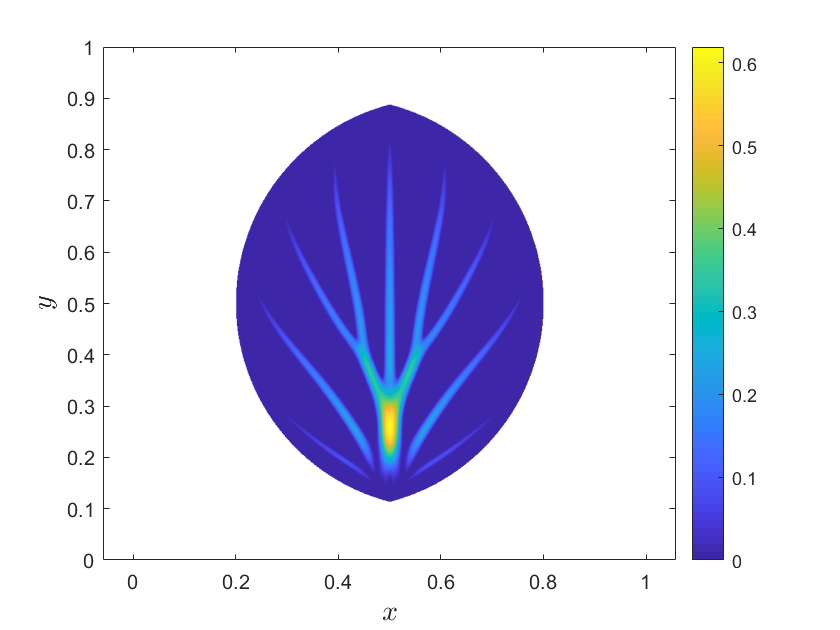}  &  \includegraphics[width=0.3\textwidth]{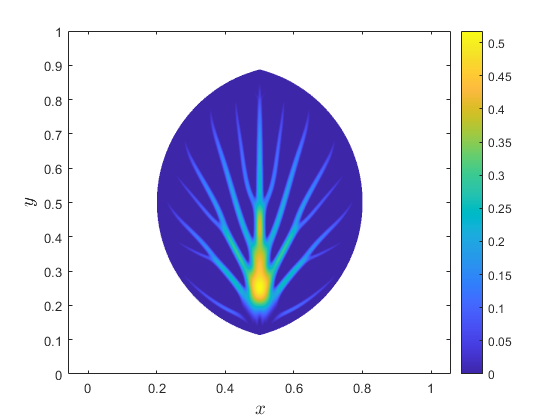}  \\ \hline
    \end{tabular}   
\caption{\textit{Comparison of the solutions decreasing the coefficient $\widetilde D$, with leaf-shaped domain.  In these plots $N = 800$  and other parameters are in Table \ref{tab:parameters}. }}
\label{fig:Dtilde}
\end{figure}

\begin{figure}[h]
    \centering
    \begin{tabular}{|c|c|c|c|}
    \hline & quartic & Fisher & mixed \\
\hline \begin{turn}{90} scalar \end{turn}
 &
    \includegraphics[width=0.3\textwidth]{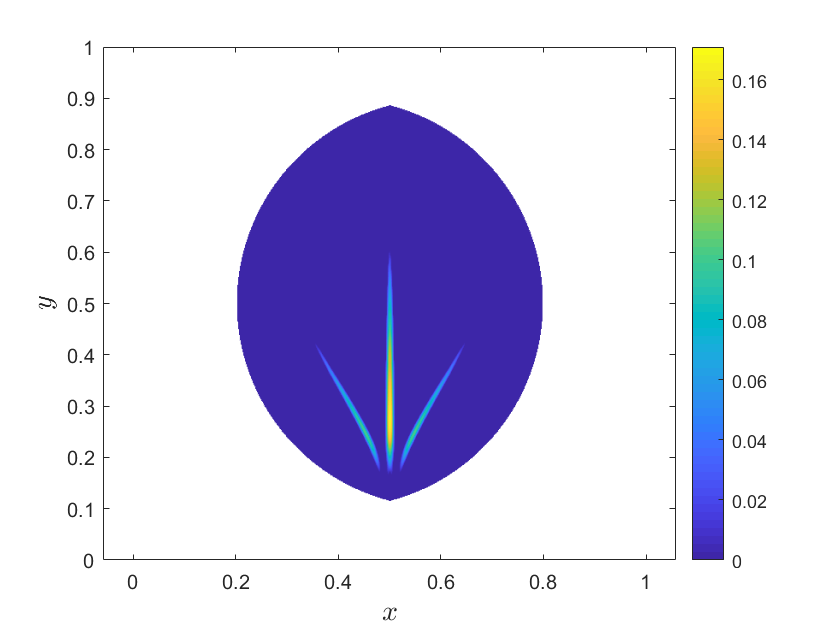} 
&
    \includegraphics[width=0.3\textwidth]{Figures_2D/solu_base_case_N800.png} 
&
    \includegraphics[width=0.3\textwidth]{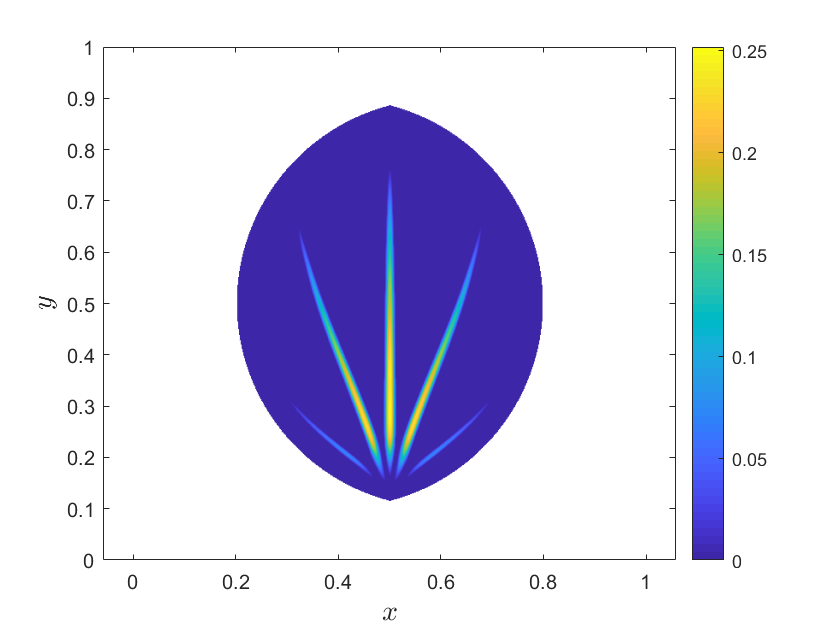}  \\
\hline \begin{turn}{90} tensor \end{turn}
 &
    \includegraphics[width=0.3\textwidth]{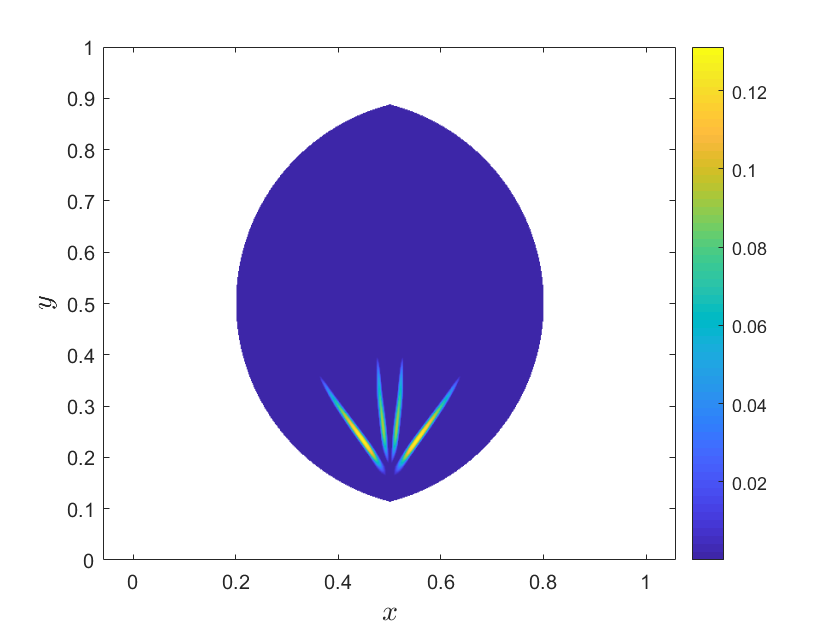} 
&
    \includegraphics[width=0.3\textwidth]{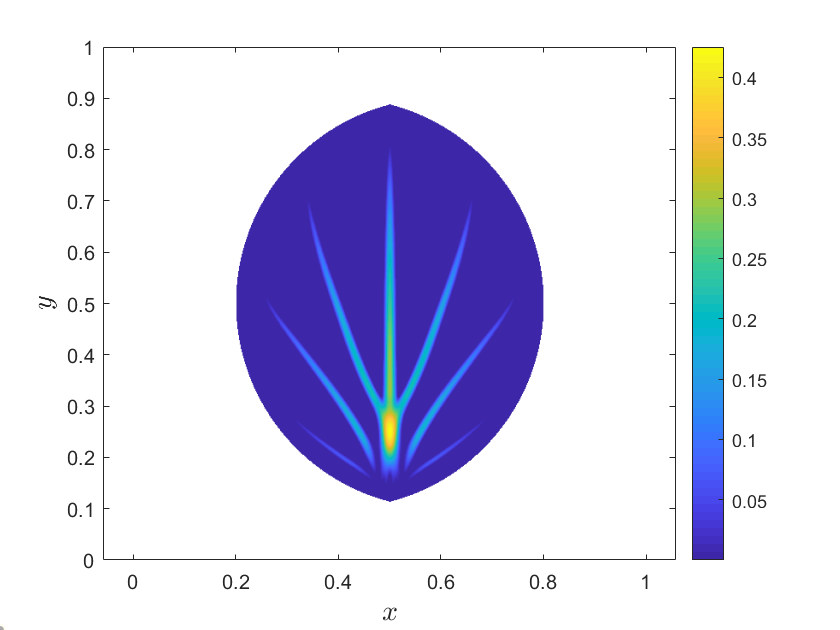} 
&
    \includegraphics[width=0.3\textwidth]{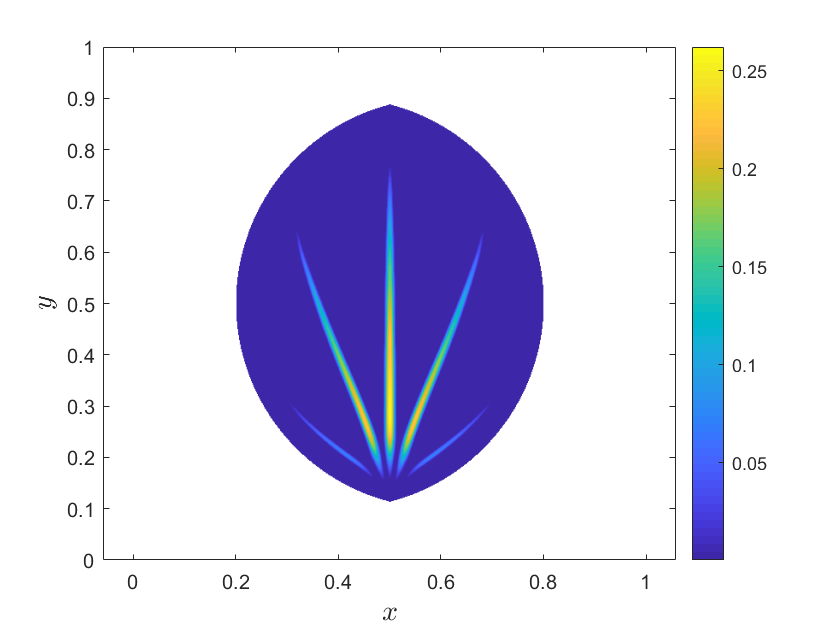}  \\
\hline    
\end{tabular}
\caption{\textit{Comparison of the solutions with leaf-shaped domain, considering different expressions of $\Phi''(p)$. For the scalar case, the solution $\mathbb{C}$ is directly plotted. For the tensor case, we represent the solution using its Frobenius norm, denoted as $||\mathbb{C}||$.
In these plots $N = 600$, while the other parameters are in Table \ref{tab:parameters}.}}
\label{fig:entropy}
\end{figure}

\begin{figure}[h]
    \begin{tabular}{|c|c|c|c|c|}
    \hline & N = 400 & N = 600 & N = 800 & N = 1000 \\
\hline \begin{turn}{90} $\mathbb C$ \end{turn} &
  \includegraphics[width=0.2\textwidth]{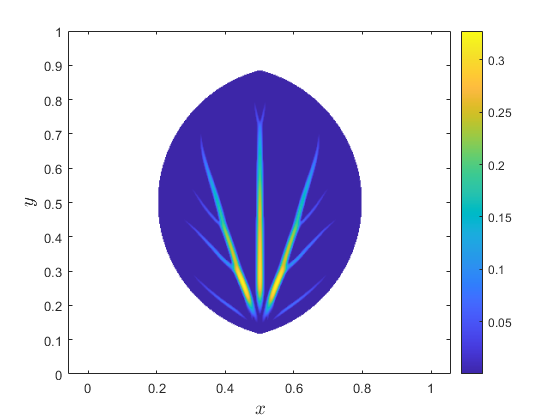}  &  \includegraphics[width=0.2\textwidth]{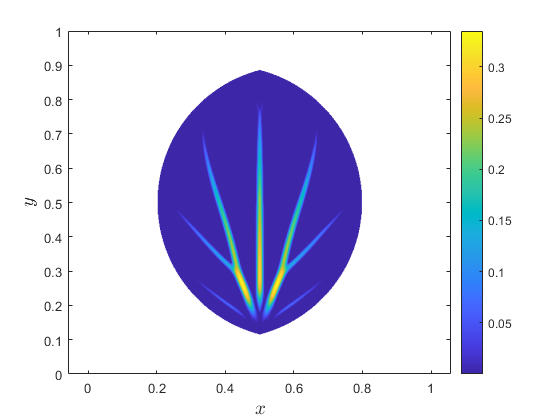} &
 \includegraphics[width=0.2\textwidth]{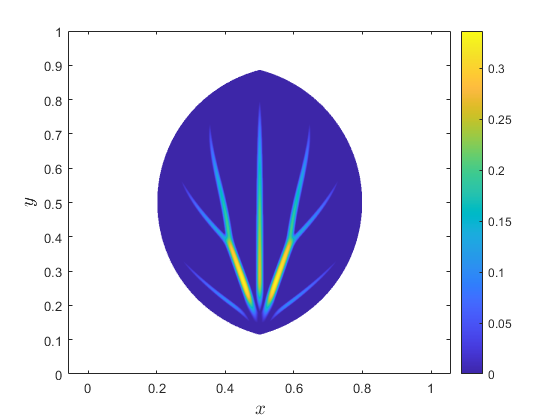} &
 \includegraphics[width=0.2\textwidth]{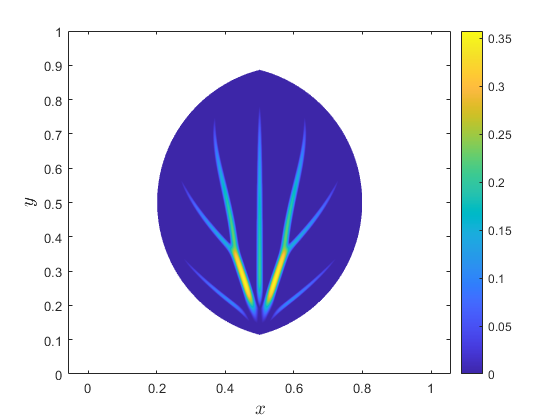} \\  \hline
  \begin{turn}{90} $\log \left( \mathbb C \right)$ \end{turn} & 
 \includegraphics[width=0.2\textwidth] {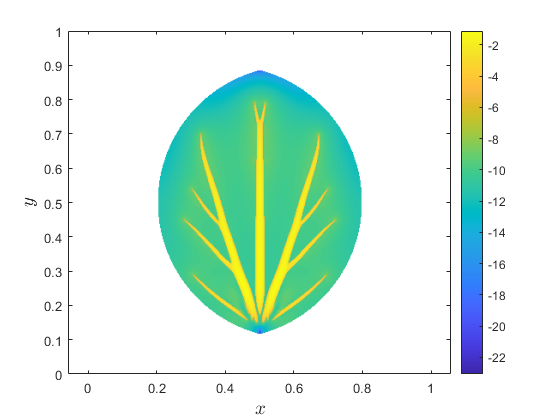}  &  \includegraphics[width=0.2\textwidth]{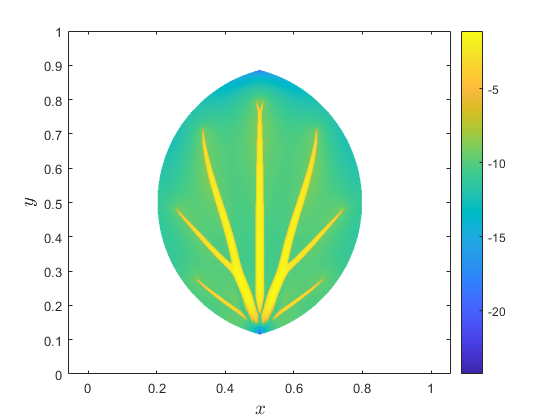} &
 \includegraphics[width=0.2\textwidth]{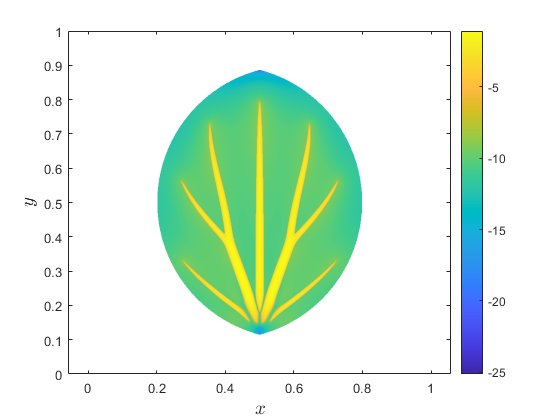} &
 \includegraphics[width=0.2\textwidth]{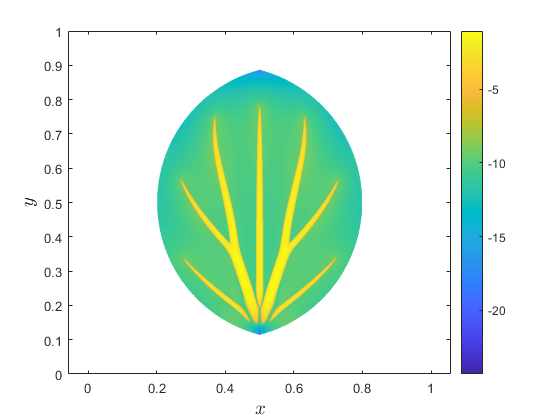}\\ \hline
 \end{tabular}
    \caption{\textit{Configuration of the numerical solution increasing the number $N$ of cells, for leaf-shaped domain. We show the solution $\mathbb C$ in the first line, and its logarithm $\log (\mathbb{C})$ in the second line.}}
    \label{fig:log_r5em3}
\end{figure}

\begin{figure}[h]
    \begin{tabular}{|c|c|c|c|}
    \hline $r$ & $\mathbb C$ & $\log{\mathbb C}$ & $\log{\mathbb C}$, zoom-in \\
\hline
 \begin{turn}{90} $\qquad r = 10^{-3}$ \end{turn}
 & \includegraphics[width=0.3\textwidth]{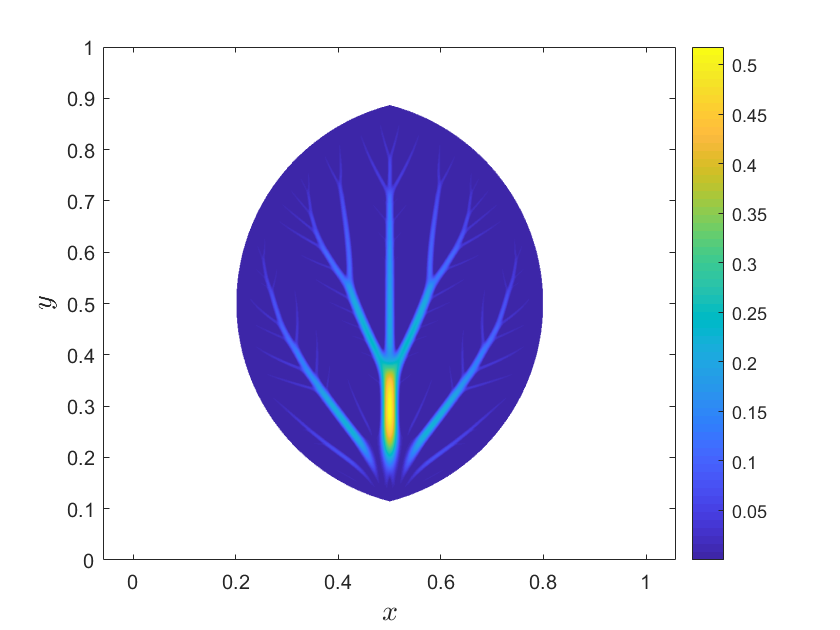} & \includegraphics[width=0.3\textwidth]{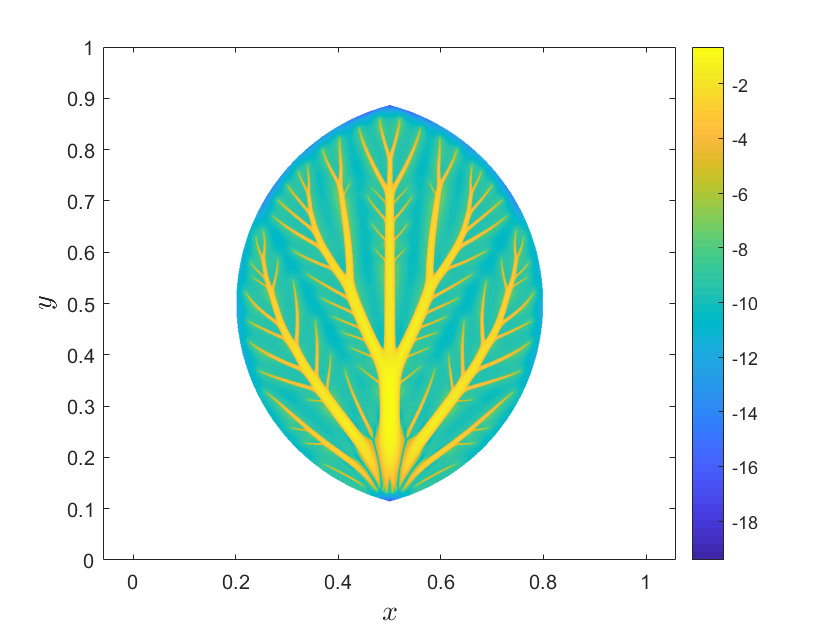}  &  \includegraphics[width=0.3\textwidth]{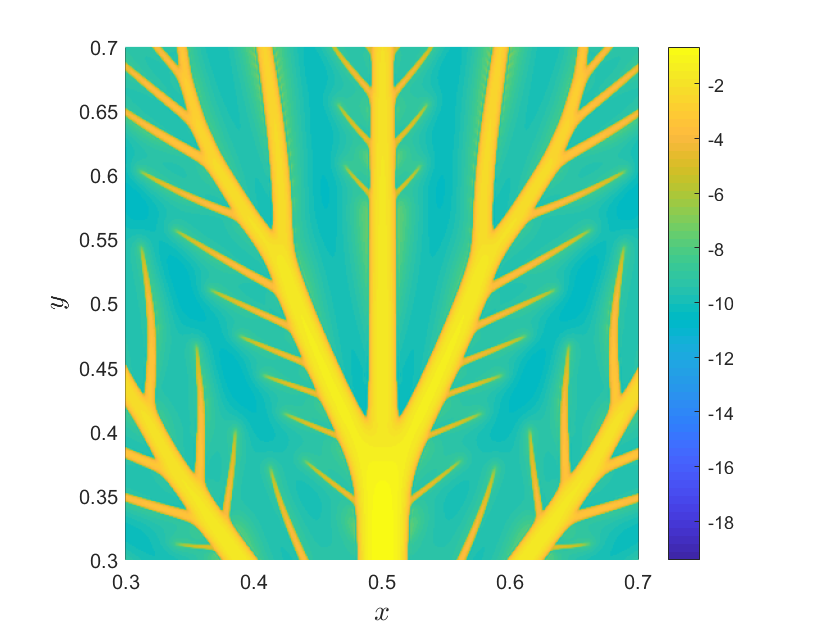}  \\  \hline \begin{turn}{90} $\qquad r = 5\cdot 10^{-4}$ \end{turn}
 & \includegraphics[width=0.3\textwidth]{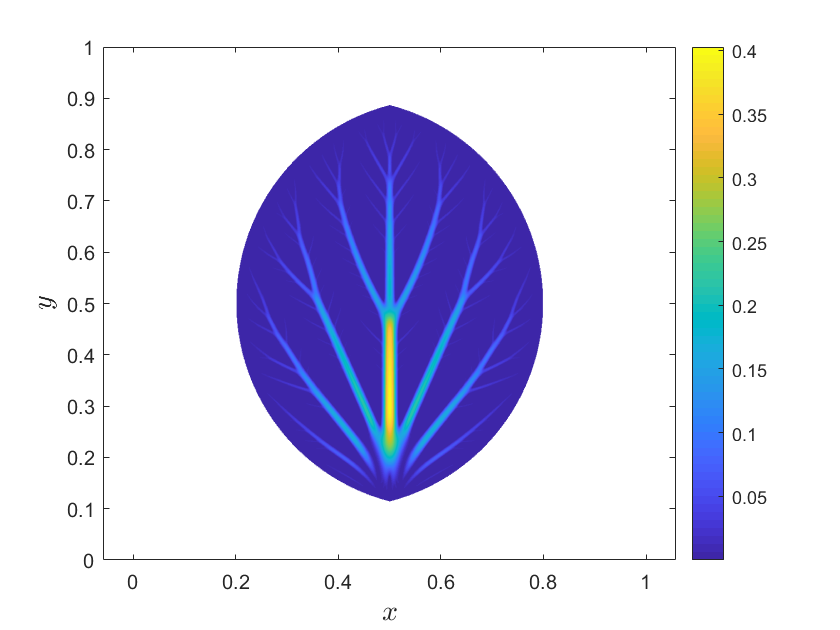} & \includegraphics[width=0.3\textwidth]{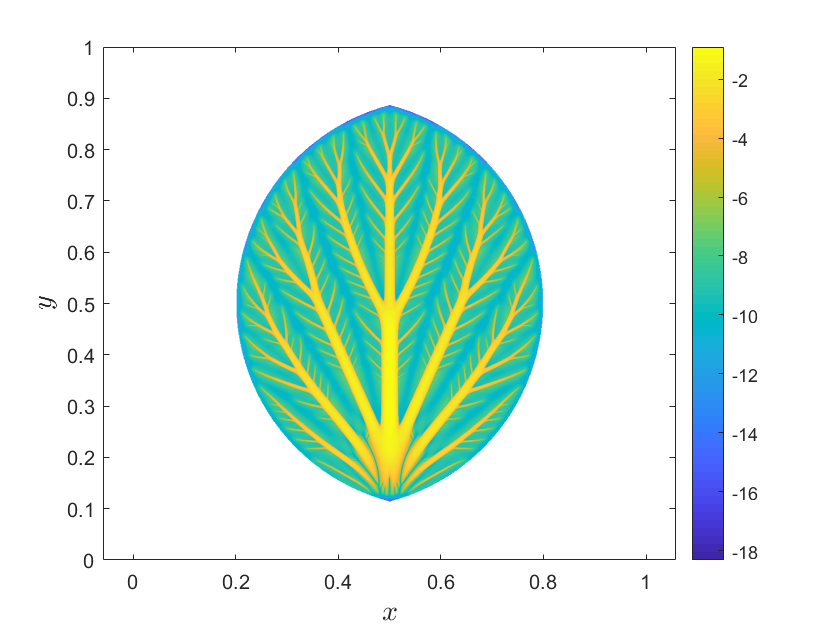}  &  \includegraphics[width=0.3\textwidth]{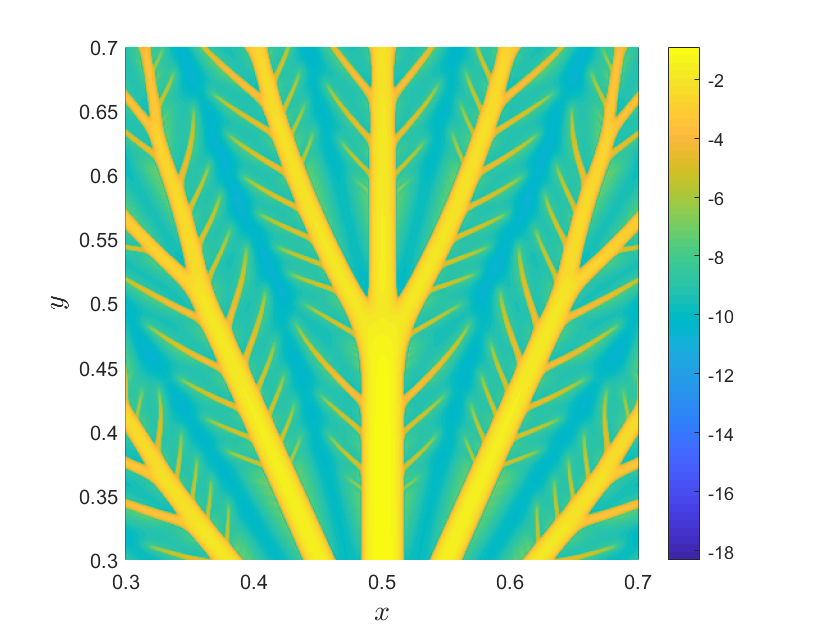}  \\ \hline
 \end{tabular}
    \caption{\textit{Behavior of the solution $\mathbb{C}$ and its logarithm $\log(\mathbb C)$ for smaller values of $r$, illustrating the increase in the number of branches increases as $r \to 0$. The other parameters are in Table \ref{tab:parameters}.}}
    \label{fig:log_r1em3}
\end{figure}

\section{Conclusions}
In this paper, we began by revisiting the $L^2$-gradient flow for a class of self-regulating processes characterized by the minimization of entropy dissipation, coupled with a conservation law for a quantity such as chemical concentration, ions, nutrients, or material pressure, firstly introduced in \cite{portaro2022emergence}. We established a local existence and uniqueness theorem within Hölder spaces, employing Schauder and semigroup theory. The development of a comprehensive Sobolev theory remains an ongoing challenge due to the absence of sufficient regularity estimates and the lack of a minimum principle for the conductivity tensor.

\begin{figure}[h]
    \begin{tabular}{|c|c|c|c|}
    \hline $r$ & $\mathbb C$ & $\log{\mathbb C}$ & contour lines \\
\hline
 \begin{turn}{90} $\qquad r = 10^{-3}$ \end{turn}
 & \includegraphics[width=0.3\textwidth]{Figures_2D/leaf_N1000_r1em3.png} & \includegraphics[width=0.3\textwidth]{Figures_2D/leaf_N1000_r1em3_log.png}  &  \includegraphics[width=0.3\textwidth]{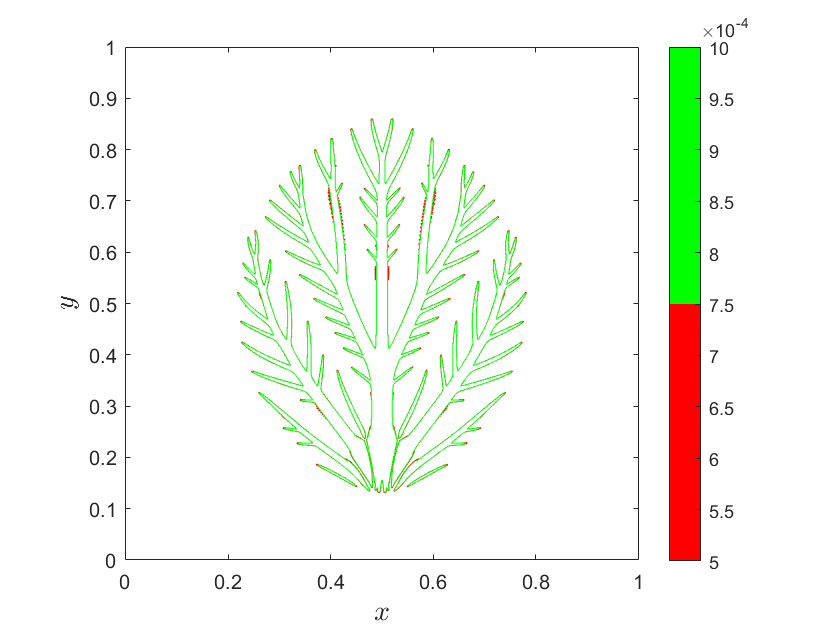}  \\  \hline \begin{turn}{90} $\qquad r = 5\cdot 10^{-4}$ \end{turn}
 & \includegraphics[width=0.3\textwidth]{Figures_2D/leaf_N1000_r5em4.png} & \includegraphics[width=0.3\textwidth]{Figures_2D/leaf_N1000_r5em4_log.png}  &  \includegraphics[width=0.3\textwidth]{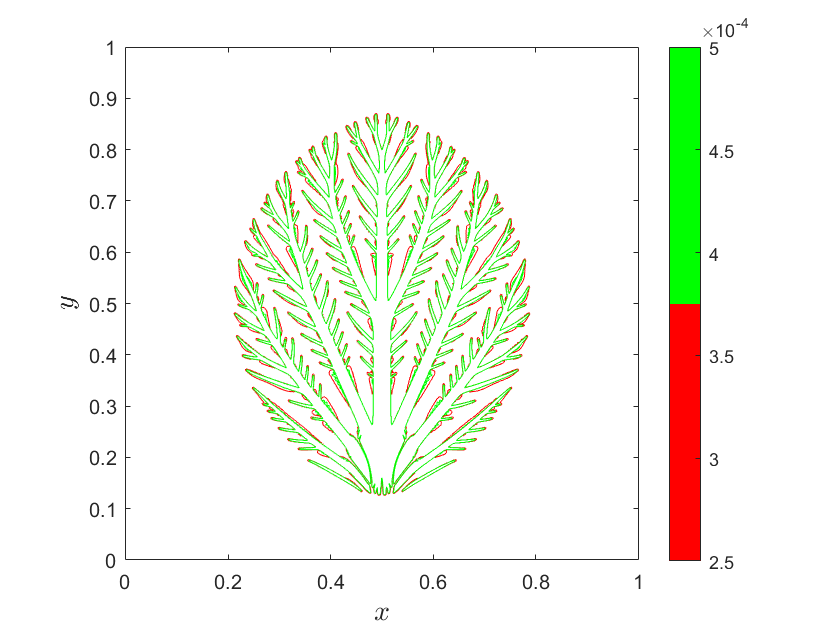}  \\ \hline
 \end{tabular}
    \caption{\textit{The first column displays the solution $\mathbb C$, the second column shows its logarithm $\log(\mathbb C)$, and the third column illustrates the contour lines for $\mathbb C = r$ (green) and $\mathbb{C} = r/2$ (red). In these tests $N = 1000$ and the other parameters are in Table \ref{tab:parameters}.}}
    \label{fig:contour_r}
\end{figure}

\begin{figure}[h]
    \begin{tabular}{|c|c|c|c|}
    \hline $N$ & $\mathbb C$ & $\log{\mathbb C}$ & $\log{\mathbb C}$, zoom-in \\
\hline \begin{turn}{90}$\qquad \qquad 400$\end{turn}
 &
\includegraphics[width=0.3\textwidth]{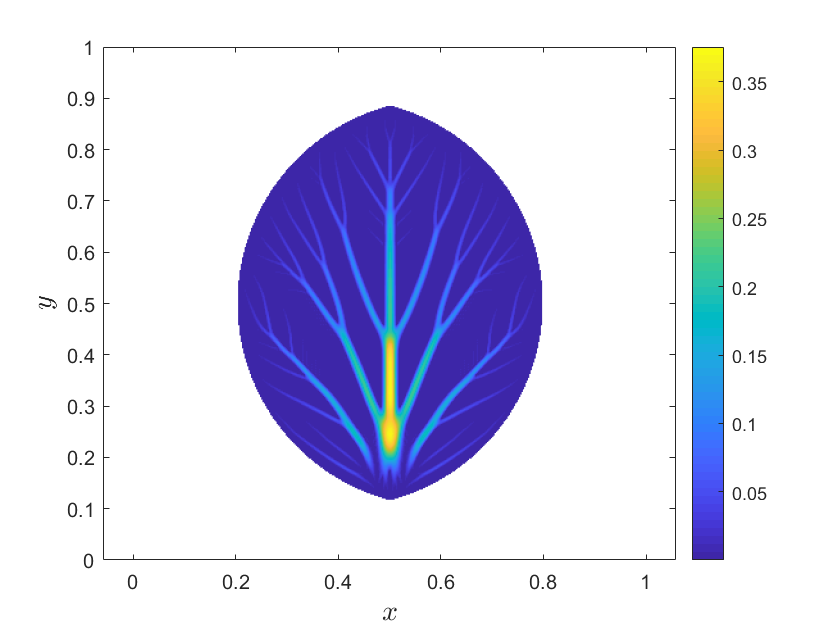}     &  \includegraphics[width=0.3\textwidth]{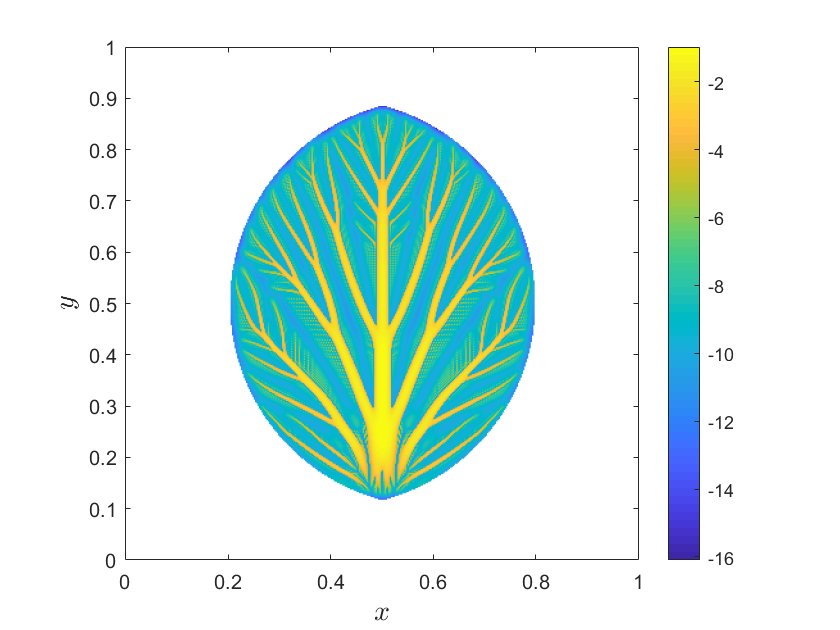} & \includegraphics[width=0.3\textwidth]{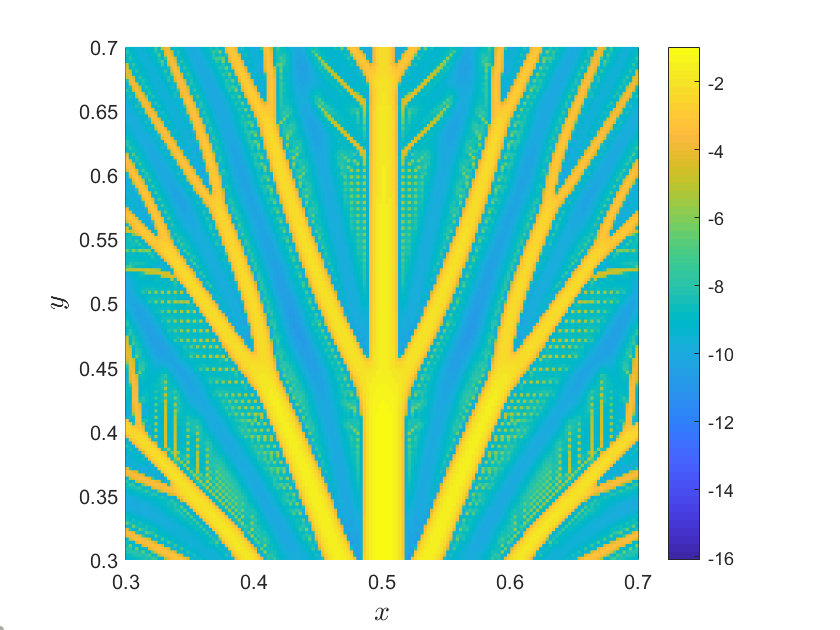}\\ \hline
    \begin{turn}{90} $\qquad \qquad 600$ \end{turn}
 & \includegraphics[width=0.3\textwidth]{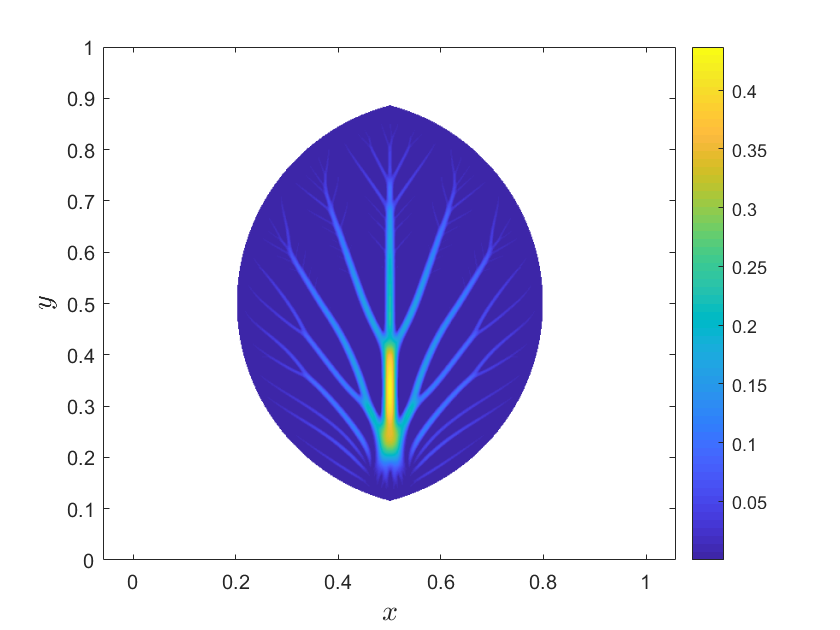} & \includegraphics[width=0.3\textwidth]{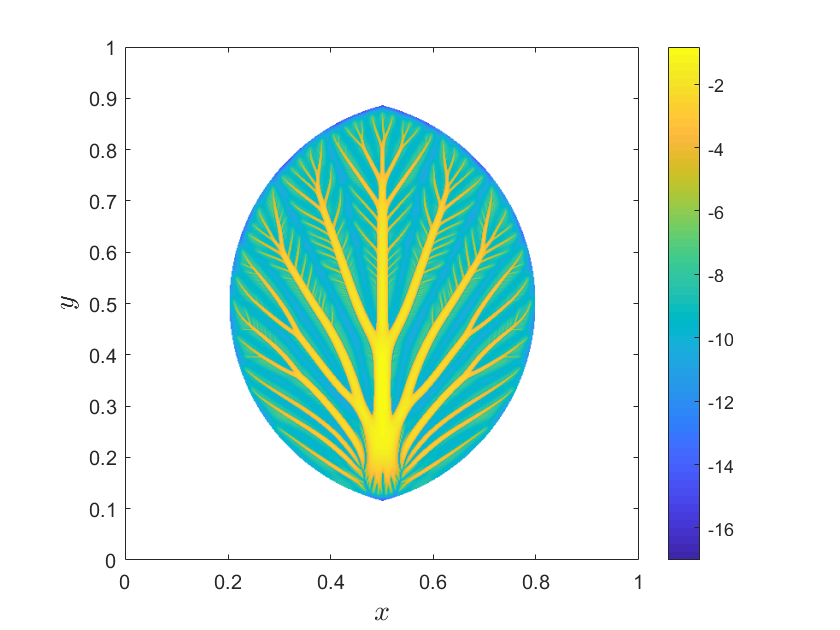}  &  \includegraphics[width=0.3\textwidth]{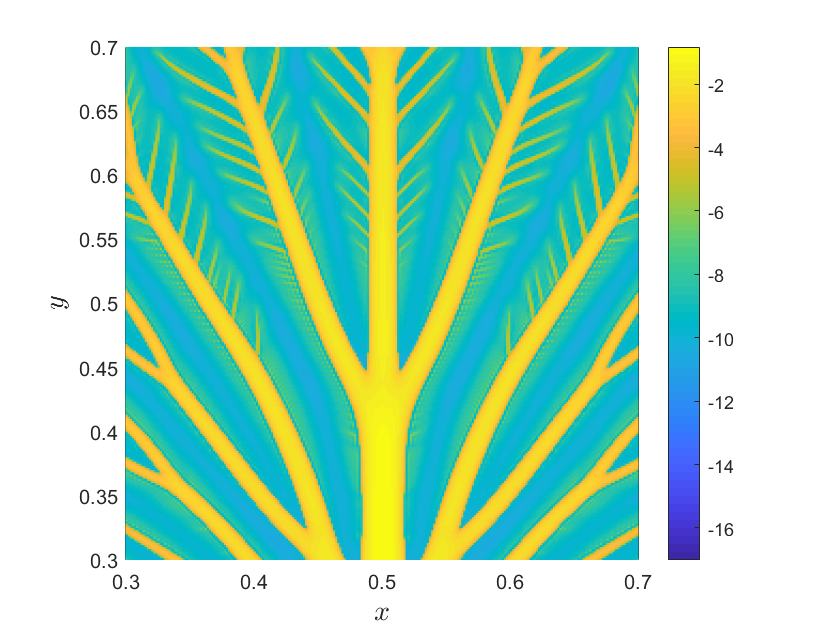}  \\ \hline
 \begin{turn}{90} $\qquad \qquad 800$ \end{turn}
 & \includegraphics[width=0.3\textwidth]{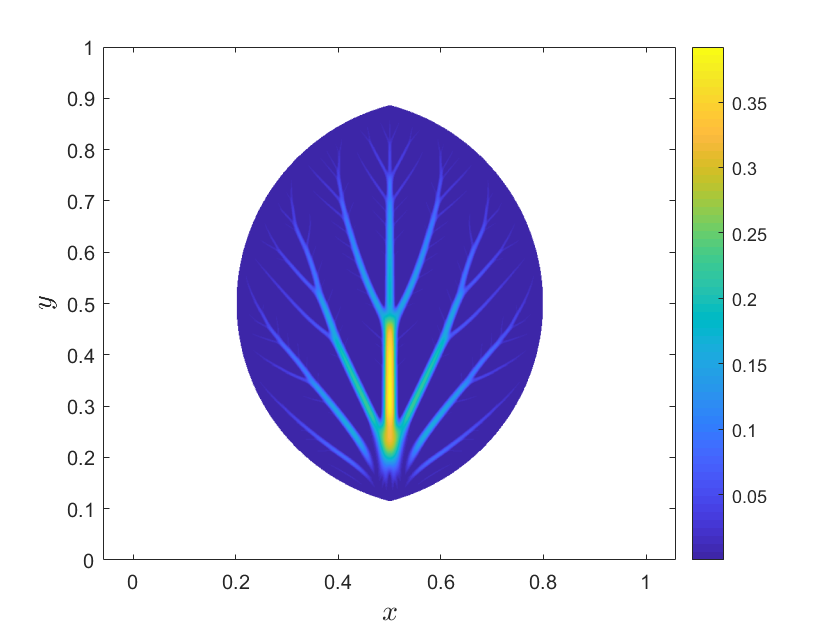} & \includegraphics[width=0.3\textwidth]{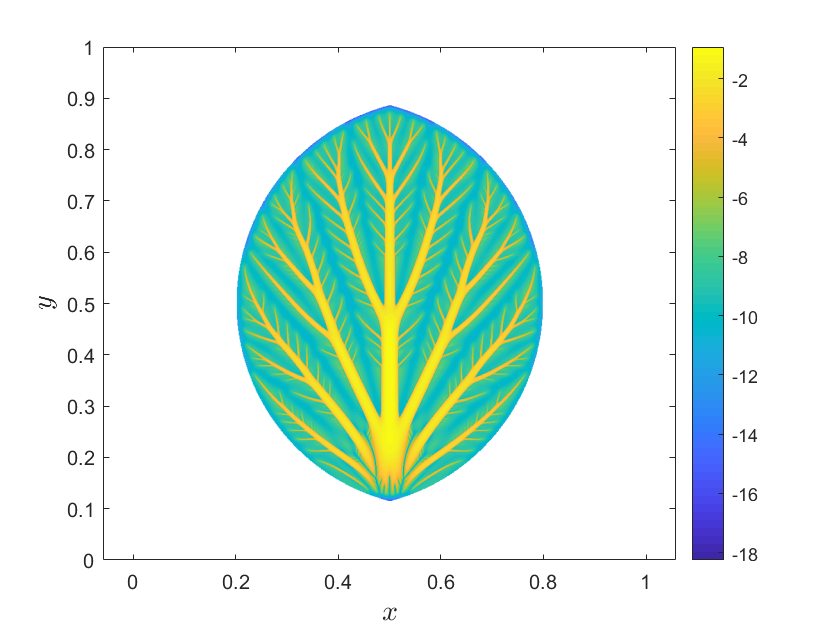}  &  \includegraphics[width=0.3\textwidth]{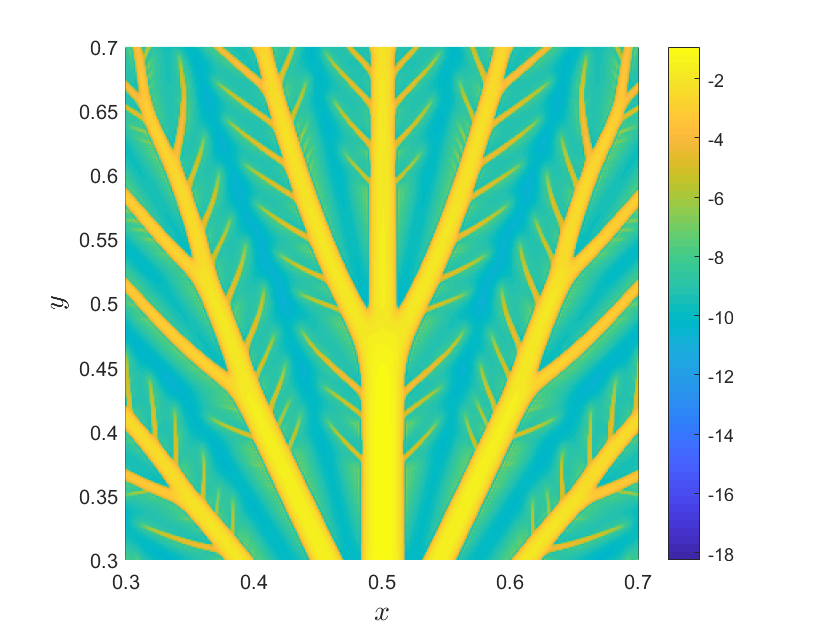}  \\ \hline
  \begin{turn}{90} $\qquad \qquad 1000$ \end{turn}
 & \includegraphics[width=0.3\textwidth]{Figures_2D/leaf_N1000_r5em4.png} & \includegraphics[width=0.3\textwidth]{Figures_2D/leaf_N1000_r5em4_log.png}  &  \includegraphics[width=0.3\textwidth]{Figures_2D/leaf_N1000_r5em4_log_zoom.png}  \\ \hline
    \end{tabular}
    \caption{\textit{Configuration of the solution and the corresponding number of branches for various values of $N$ and $r = 10^{-4}$. The other parameters are in Table \ref{tab:parameters}.}}
    \label{fig:log_r5em4}
\end{figure}


A key objective of our research was to conduct extensive numerical experiments. We were interested in showing results in different geometries and for different sets of parameters. 
To this purpose we adopted a newly developed nodal ghost method, which is particularly suited for evolutionary PDEs on arbitrary domains in two space dimensions \cite{astuto2024nodal}. 
We began with a circular domain to establish that the numerical solution is barely affected by grid orientation. To further validate this, we conducted tests on a leaf-shaped domain, including showing the effects of a rotated domain.

The main discover of the paper is the following. As $r$ decreases, the solution becomes increasingly intricate, with numerous small branches spreading across a larger area of the computational domain. When $\mathbb C < r$, isotropic permeability takes over, causing the branches to disappear. Moreover, as 
$r$ gets smaller, finer scales emerge in the solution, making it more challenging to achieve accurate results due to the need for higher number of cells to capture these details. A parallel version of the code, which allows unprecedented high resolution computations, is currently in preparation. The new code should allow us to formulate conjectures about the fractal nature of the small branches for vanishingly small values of the isotropic conductivity, for the various models considered in this paper.

\section*{Acknowledgements}
C.A. and G.R. are members of the Gruppo Nazionale Calcolo Scientifico-Istituto Nazionale di Alta Matematica (GNCS-INdAM).

C.A. is supported by the FAIR Spoke 10 Project under the National Recovery and Resilience Plan (PNRR).

\bibliographystyle{plain}
\bibliography{bibliography}

\end{document}